\documentclass{amsart}

\newtheorem{theorem}{Theorem}[section]
\newtheorem{lemma}[theorem]{Lemma}
\newtheorem{proposition}[theorem]{Proposition}
\newtheorem{corollary}[theorem]{Corollary}

\newtheorem{problem}[theorem]{Problem}

\theoremstyle{definition}
\newtheorem{definition}[theorem]{Definition}

\newtheorem{remark}[theorem]{Remark}

\def\cocoa{{\hbox{\rm C\kern-.13em o\kern-.07em C\kern-.13em o\kern-.15em A}}}

\usepackage{amscd,amssymb}
\usepackage{youngtab}
\begin{document}

\title[Weak polynomial identities of $3\times 3$ skew-symmetric matrices]
{Cocharacters for the weak polynomial identities of the Lie algebra of $3\times 3$ skew-symmetric matrices}

\author[M\'aty\'as Domokos and Vesselin Drensky]
{M\'aty\'as Domokos and Vesselin Drensky}
\address{Alfr\'ed R\'enyi Institute of Mathematics,
Re\'altanoda utca 13-15, 1053 Budapest, Hungary}
\email{domokos.matyas@renyi.hu}
\address{Institute of Mathematics and Informatics,
Bulgarian Academy of Sciences,
Acad. G. Bonchev Str., Block 8,
1113 Sofia, Bulgaria}
\email{drensky@math.bas.bg}

\thanks{This project was carried out in the framework of the exchange program
between the Hungarian and Bulgarian Academies of Sciences.
It was partially supported by the Hungarian National Research, Development and Innovation Office,  NKFIH K 119934,  K 132002, 
and by Grant KP-06 N 32/1 of 07.12.2019 ``Groups and Rings - Theory and Applications'' of the Bulgarian National Science Fund.}
\thanks{Corresponding author: M. Domokos}

\subjclass[2010]{16R10; 16R30; 17B01; 17B20; 17B45; 20C30.}

\keywords{weak polynomial identities, skew-symmetric matrices, classical invariant theory}

\maketitle

\begin{abstract}
Let $so_3(K)$ be the Lie algebra of $3\times 3$ skew-symmetric matrices over a field $K$ of characteristic 0.
The ideal $I(M_3(K),so_3(K))$ of the weak polynomial identities of the pair $(M_3(K),so_3(K))$ consists of the elements $f(x_1,\ldots,x_n)$
of the free associative algebra $K\langle X\rangle$ with the property that $f(a_1,\ldots,a_n)=0$
in the algebra $M_3(K)$ of all $3\times 3$ matrices for all $a_1,\ldots,a_n\in so_3(K)$.
The generators of $I(M_3(K),so_3(K))$ were found by Razmyslov in the 1980s.
In this paper the cocharacter sequence of $I(M_3(K),so_3(K))$ is computed.
In other words, the ${\mathrm{GL}}_p(K)$-module structure of the algebra generated by $p$ generic skew-symmetric matrices is determined.
Moreover, the same is done for the closely related algebra of $\mathrm{SO}_3(K)$-equivariant polynomial maps
from the space of $p$-tuples of $3\times 3$ skew-symmetric matrices into $M_3(K)$ (endowed with the conjugation action).
In the special case $p=3$ the latter algebra is a module over a $6$-variable polynomial subring
in the algebra of $\mathrm{SO}_3(K)$-invariants of triples of $3\times 3$ skew-symmetric matrices, and a free resolution of this module is found.
The proofs involve methods and results of classical invariant theory, representation theory of the general linear group
and explicit computations with matrices.
\end{abstract}

\section{Introduction} \label{sec:intro}

This paper can be considered as a relative of the well-known paper of Procesi \cite{Procesi:2} whose abstract says that
``In a precise way the ring of $m$ generic $2\times 2$ matrices and related rings are described.''  In the present work we also describe
the ring of $m$ generic $3\times 3$ skew-symmetric matrices and a related ring
in a precise way, but in somewhat different
terms than \cite{Procesi:2} (and we restrict to the case of a characteristic zero base field).

Take $3\times 3$ generic skew-symmetric matrices
\[
t_k=\left(\begin{matrix}
0&t^{(k)}_{12}&t_{13}^{(k)}\\
-t_{12}^{(k)}&0&t_{23}^{(k)}\\
-t^{(k)}_{13}&-t^{(k)}_{23}&0\\
\end{matrix}\right),\quad k=1,\ldots,p,
\]
where $T_p=\{t^{(k)}_{ij}\mid i,j=1,2,3;\  k=1,\ldots,p\}$ are commuting variables.
Till the end of the paper we fix a field $K$ of characteristic zero.
Then $t_1,\dots,t_p$ are elements of the $3\times 3$ matrix algebra $M_3(K[T_p])$ over the polynomial ring $K[T_p]$.
As usual, we shall identify the elements of $K[T_p]$ with polynomial maps
$so_3(K)^{\oplus p}\to K$ in the obvious way, where $so_3(K)$ is the space of $3\times 3$ skew-symmetric matrices over $K$,
which is also the Lie algebra of the special orthogonal group $\mathrm{SO}_3(K)=\{A\in K^{3\times 3}\mid AA^T=I,\ \det(A)=1\}$.
Accordingly, $M_3(K[T_p])$ is identified with the set of polynomial maps $so_3(K)^{\oplus p}\to M_3(K)$.
Denote by  ${\mathcal{F}}_p$ the associative $K$-subalgebra (with an identity element) of $M_3(K[T_p])$ generated by
$t_1,\dots,t_p$ (so the identity matrix $I$ is an element of ${\mathcal{F}}_p$ by definition):
\[
{\mathcal{F}}_p=K\langle t_1,\dots,t_p\rangle\subset M_3(K[T_p])
\]
The special orthogonal group $\mathrm{SO}_3(K)$ acts on $so_3(K)$ by conjugation (the adjoint action of $\mathrm{SO}_3(K)$ on its Lie algebra),
and $\mathrm{SO}_3(K)$ acts by simultaneous conjugation
on $so_3(K)^{\oplus p}$, the space of $p$-tuples of skew-symmetric $3\times 3$ matrices.
Also $\mathrm{SO}_3(K)$ acts on $M_3(K)$ by conjugation,
and we write ${\mathcal{E}}_p$ for the subset of $M_3(K[T_p])$ consisting of the $\mathrm{SO}_3(K)$-equivariant polynomial maps
$so_3(K)^{\oplus p}\to M_3(K)$. Clearly ${\mathcal{E}}_p$ is an associative $K$-subalgebra of
$M_3(K[T_p])$, and ${\mathcal{E}}_p$ contains ${\mathcal{F}}_p$.  If follows easily from known results
(see  Corollary~\ref{cor:cov-generators}) that
\[
{\mathcal{E}}_p=K\langle t_1,\dots,t_p,\ \mathrm{tr}(t_it_j) I,\ \mathrm{tr}(t_kt_lt_m) I\mid
i\le j,\ k<l<m \rangle\subset M_3(K[T_p]),
\]
where $I$ stands for  the $3\times 3$ identity matrix throughout the paper.

In the present paper we aim at a combinatorial description of the algebras ${\mathcal{F}}_p$ and ${\mathcal{E}}_p$.
The general linear group ${\mathrm{GL}}_p(K)$ acts on ${\mathcal{E}}_p$ via graded
$K$-algebra automorphisms.
Note first that ${\mathrm{GL}}_p(K)$ acts (from the right) on $so_3(K)^{\oplus p}$ as follows. For
\[
g=\left(\begin{matrix}
g_{11}&\ldots&g_{1p}\\
\vdots&\ddots&\vdots\\
g_{p1}&\ldots&g_{pp}\\
\end{matrix}\right),\text{ and  }a=(a_1,\dots,a_p)\in so_3(K)^{\oplus p}
\]
we have
\[
(a_1,\dots,a_p)\cdot g=(\sum_{i=1}^pg_{i1}a_i,\sum_{i=1}^pg_{i2}a_2,
\dots,\sum_{i=1}^pg_{ip}a_i).
\]
This induces a left  action (via graded $K$-algebra automorphisms) of ${\mathrm{GL}}_p(K)$
on the algebra $K[T_p]$ (respectively $M_3(K[T_p])$) of polynomial maps
$so_3^{\oplus p}\to K$ (respectively $so_3^{\oplus p}\to M_3(K[T_p])$) in the standard way
(for a function $f$, we have $(g\cdot f)(a)=f(a\cdot g)$).
More explicitly, $g\cdot t_{ij}^{(k)}=\sum_{l=1}^pg_{lk}t_{ij}^{(l)}$, and for a matrix
$m=(m_{ij})_{i,j=1}^3\in M_3(K[T_p])$, we have $g\cdot m=(g\cdot m_{ij})_{i,j=1}^3$.
The ${\mathrm{GL}}_3(K)$-action on $so_3(K)^{\oplus p}$ commutes with the
$\mathrm{SO}_3(K)$-action, hence ${\mathcal{E}}_p$ is a ${\mathrm{GL}}_p(K)$-submodule of $M_3(K[T_p])$.
Obviously ${\mathcal{F}}_p$ is a ${\mathrm{GL}}_p(K)$-submodule in ${\mathcal{E}}_p$.

We shall determine the ${\mathrm{GL}}_p(K)$-module structure both for ${\mathcal{F}}_p$ and ${\mathcal{E}}_p$.
Our Theorem~\ref{main theorem}
(see also Theorem~\ref{thm:main theorem 2}) and Theorem~\ref{thm:covariant-GL} (i) 
(together with Lemma~\ref{semistandard tableaux from 5 to 3}) give the multiplicities of the irreducible ${\mathrm{GL}}_p(K)$-modules
as summands in ${\mathcal{F}}_p$ and in ${\mathcal{E}}_p$.
In fact, in the course of the proofs  highest weight vectors for each irreducible summand are explicitly provided.
In the case of ${\mathcal{E}}_p$ results from classical invariant theory allow to compute these multiplicities.
Then with explicit constructions we show that for almost all irreducibles these upper bounds are achieved even in ${\mathcal{F}}_p$.
In particular, our Theorem~\ref{thm:covariant-GL} (ii)
shows exactly the difference between ${\mathcal{E}}_p$ and its subspace ${\mathcal{F}}_p$; namely,
the ${\mathrm{GL}}_p(K)$-module ${\mathcal{E}}_p$ has the direct sum decomposition
\[
{\mathcal{E}}_p={\mathcal{F}}_p\oplus \bigoplus_{k=1}^{\infty}\langle\mathrm{tr}(t_1^{2k})I\rangle_{{\mathrm{GL}}_p(K)}.
\]
Here and later as well,
given a subset $U$ of a ${\mathrm{GL}}_p(K)$-module we write $\langle U\rangle_{{\mathrm{GL}}_p(K)}$ for the ${\mathrm{GL}}_p(K)$-submodule generated by $U$.
For $p\le q$, ${\mathcal{E}}_p$ is a subalgebra of ${\mathcal{E}}_q$ and ${\mathcal{F}}_p$ is a subalgebra of
${\mathcal{F}}_q$. It follows from general principles that for $p\ge 3$,
we have
\[
{\mathcal{F}}_p=\langle {\mathcal{F}}_3\rangle_{{\mathrm{GL}}_p(K)} \text{ and }
{\mathcal{E}}_p=\langle {\mathcal{E}}_3 \rangle_{{\mathrm{GL}}_p(K)}
\]
(see Section~\ref{subsec:wpi} and Corollary~\ref{cor:weyl}).
Therefore to a large extent, the combinatorial study of ${\mathcal{E}}_p$ and ${\mathcal{F}}_p$ can be reduced to the special case $p=3$.
We shall present the $3$-variable Hilbert series (i.e. the formal ${\mathrm{GL}}_3(K)$-character) of ${\mathcal{E}}_3$ as a rational function
(see Proposition~\ref{prop:hilbert_series}).
Furthermore, ${\mathcal{E}}_3$ is a module over the algebra $K[T_3]^{\mathrm{SO}_3(K)}$ of polynomial $\mathrm{SO}_3(K)$-invariants on $so_3(K)$,
and we shall determine the structure of this module
(see Theorem~\ref{thm:cov_3}).

The algebra ${\mathcal{F}}_p$ is isomorphic to the factor of the free associative algebra
$K\langle X_p\rangle =K\langle x_1,\dots,x_p\rangle$
modulo $I(M_3(K),so_3(K))\cap K\langle X_p\rangle$, the ideal of $p$-variable weak polynomial identities of the pair $(M_3(K),so_3(K))$.
Therefore the computation of the ${\mathrm{GL}}_p(K)$-module structure of ${\mathcal{F}}_p$
is the same thing as the computation of the cocharacter sequence of the ideal $I(M_3(K),so_3(K))$
of weak polynomial identities of the pair $(M_3(K),so_3(K))$.
In fact this was our original motivation for the present work, since weak polynomial identities play a significant role in the theory of PI-algebras.
An overview of some relevant results on weak polynomial identities is given in
Section~\ref{subsec:wpi}. 

We note that our computation is independent of the base field and the \emph{form} of the Lie algebra. In particular, Theorem~\ref{main theorem} can be interpreted as the 
computation of the cocharacter sequence for the weak polynomial identities of the pair 
$(M_3(K),\mathrm{ad}(sl_2(K))$, where $\mathrm{ad}$ stands for the adjoint representation of the Lie algebra $sl_2(K)$.

\section{Preliminaries}\label{sec:prel}

For a background on the mathematics used in this paper we recommend:
\begin{itemize}
\item On trace identities the paper by Procesi \cite{Procesi:1}
and the book by Razmyslov \cite[Chapter IV]{Razmyslov:4};

\item On invariant theory the book by Weyl \cite{Weyl};

\item On representation theory of the general linear group the book by Macdonald \cite{Macdonald}
and for the applications to algebras with polynomial identities the book by one of the authors \cite[Chapter 12]{Drensky:2}.
\end{itemize}

\subsection{Weak polynomial identities} \label{subsec:wpi}

\begin{definition}
Let $R$ be an associative algebra over a field $K$ and let $R^{(-)}$
be the Lie algebra with respect to the operation  $[r_1,r_2]=r_1r_2-r_2r_1$, $r_1,r_2\in R$.
Let $L$ be a Lie subalgebra of $R^{(-)}$ which generates $R$ as an associative algebra,
i.e., $R$ is an associative enveloping algebra of $L$.
The polynomial $f(x_1,\ldots,x_n)$ of the free associative algebra
$K\langle X\rangle=K\langle x_1,x_2,\ldots\rangle$
is called a weak polynomial identity for the pair $(R,L)$ if
$f(a_1,\ldots,a_n)=0$ in $R$ for all $a_1,\ldots,a_n\in L$.
The ideal $I(R,L)$ of the weak polynomial identities of $(R,L)$ is generated
by the system $B=\{f_j(x_1,\ldots,x_{n_j})\mid j\in J\}$
(and $B$ is called a basis of the weak polynomial identities of the pair $(R,L)$)
if $I(R,L)$ is the minimal ideal of weak polynomial identities containing $B$.
Then $I(R,L)$ is generated as an ideal by the polynomials
$f_j(u_1,\ldots,u_{n_j})$, $j\in J$, where $u_1,\ldots,u_{n_j}$ are Lie elements in $K\langle X\rangle$. 
We shall also use the expression `$f=g$ is a weak polynomial identity for $(R,L)$' with  some 
$f,g\in K\langle X\rangle$ if $f-g\in I(R,L)$.  
\end{definition}

Weak polynomial identities were introduced by Razmyslov \cite{Razmyslov:1, Razmyslov:2}
as a powerful tool in the solution of two important problems in the theory of PI-algebras.
In \cite{Razmyslov:1} Razmyslov found bases, over a field $K$ of characteristic 0,
of the weak polynomial identities of the pair $(M_2(K),sl_2(K))$,
the polynomial identities of the Lie algebra $sl_2(K)$ of traceless $2\times 2$ matrices,
and the polynomial identities of the associative algebra $M_2(K)$ of $2\times 2$ matrices.
Up till now, in the case of characteristic 0,
the algebras $sl_2(K)$ and $M_2(K)$ are the only nontrivial simple Lie and associative algebras
with known bases of their polynomial identities.
(Another proof for the basis of the weak polynomial identities of $(M_2(K),sl_2(K))$
is given in \cite{Drensky-Koshlukov}).
In \cite{Razmyslov:2} Razmyslov constructed, using weak polynomial identities of the pair $(M_q(K),sl_q(K))$, a central polynomial
for the algebra $M_q(K)$ of $q\times q$ matrices, solving an old problem of Kaplansky \cite{Kaplansky:1, Kaplansky:2}.
The existence of central polynomials for $M_q(K)$ was established independently with other methods by Formanek \cite{Formanek}. 
(For more information on the polynomial identities and central polynomials
for matrices see, e.g., \cite{Drensky:2, Drensky-Formanek}.)

Let ${\mathfrak g}\cong sl_2({\mathbb C})$ be the three-dimensional complex simple Lie algebra
and let $U({\mathfrak g})$ be its universal enveloping algebra.
In \cite{Razmyslov:3} Razmyslov showed that the ideal $I(U({\mathfrak g}),{\mathfrak g})$ satisfies the Specht property: 
it is finitely generated and the same holds for any ideal of weak polynomial identities which contains it.
Later, in \cite[Theorem 38.1]{Razmyslov:4} (page 251 in the Russian original and page 181 in the English translation)
he found an explicit basis of the weak polynomial identities of the pair $(M_q({\mathbb C}),\varrho({\mathfrak g}))$, where
$\varrho:{\mathfrak g}\to\text{End}_{\mathbb C}(V_q)\cong M_q({\mathbb C})$ is a $q$-dimensional irreducible representation of $\mathfrak g$.
The basis consists of three weak polynomial identities:
\[
s_3(x_1,x_2,x_3)x_4=x_4s_3(x_1,x_2,x_3),
\]
where
\[
s_3(x_1,x_2,x_3):=\sum_{\sigma\in S_3}\text{sign}(\sigma)x_{\sigma(1)}x_{\sigma(2)}x_{\sigma(3)}
\]
is the standard polynomial of degree 3,
\[
\delta\sum_{\sigma\in S_3}\text{sign}(\sigma)[x_4,x_{\sigma(1)},x_{\sigma(2)},x_{\sigma(3)}]=2x_4s_3(x_1,x_2,x_3),
\]
where the commutators are left normed, e.g., $[x_1,x_2,x_3]=[[x_1,x_2],x_3]$,
and $\delta=(q^2-1)/4$ is the value of the Casimir element in the representation $\varrho$,
and one more identity in two variables
\[
\text{ART}_q(x_1,x_2):=\text{ad}x_2\prod_{i=1}^{q-1}\left(L_{x_2}-\left(i-\frac{1-q}{2}\right)\text{ad}x_2\right)x_1=0.
\]
Here $L_r:R\to R$, $r\in R$, is the operator of left multiplication of the algebra $R$, defined by $r'\to rr'$, $r'\in R$,
and $\text{ad}r(r')=[r,r']$, $r,r'\in R$.
For $q=2$ this gives that the weak polynomial identities of the pair $(M_2({\mathbb C}),sl_2({\mathbb C}))$ follow from the weak identity $[x_1^2,x_2]=0$,
which was established already in \cite{Razmyslov:1}.
The Lie algebra $sl_2({\mathbb C})$ is isomorphic to the Lie algebra $so_3({\mathbb C})$
of $3\times 3$ skew-symmetric matrices and after easy computations the result from \cite[Theorem 38.1]{Razmyslov:4} gives:

\begin{theorem}\label{theorem of Razmyslov}
The weak polynomial identities of the pair $(M_3({\mathbb C}),so_3({\mathbb C}))$ follow from its weak polynomial identities
\[
s_3(x_1,x_2,x_3)x_4=x_4s_3(x_1,x_2,x_3),
\]
\[
\sum_{\sigma\in S_3}\text{\rm sign}(\sigma)[x_4,x_{\sigma(1)},x_{\sigma(2)},x_{\sigma(3)}]=x_4s_3(x_1,x_2,x_3),
\]
and
\[
x_1[x_1,x_2]x_1=0.
\]
\end{theorem}
(The result in \cite[Theorem 38.1]{Razmyslov:4} gives explicit bases of the weak polynomial identities also in the
infinite dimensional cases.)

As in the case of ordinary polynomial identities the symmetric group $S_n$ of degree $n$ acts from the left on the vector space
$P_n\subset K\langle X\rangle$ of multilinear polynomials of degree $n$ and for any ideal $I(R,L)$ of weak polynomial identities
$P_n\cap I(R,L)$ is an $S_n$-submodule of $P_n$. The sequence of $S_n$-characters $\chi_n(R,L)$ of
$P_n/(P_n\cap I(R,L))$, $n=0,1,2,\ldots$, is called the cocharacter sequence of $I(R,L)$. Then
\[
\chi_n(R,L)=\sum_{\lambda\vdash n}m_{\lambda}\chi_{\lambda},
\]
where $\chi_{\lambda}$ is the irreducible character of $S_n$ indexed by the partition $\lambda$ of $n$
and the nonnegative integer $m_{\lambda}$ is the multiplicity of $\chi_{\lambda}$ in $\chi_n(R,L)$.
By a result of Berele \cite{Berele} and one of the authors \cite{Drensky:1} the multiplicity $m_{\lambda}$, $\lambda=(\lambda_1,\ldots,\lambda_p)\vdash n$,
is the same as the multiplicity of the irreducible polynomial $\text{GL}_p(K)$-module $W_p(\lambda)$ in the $\text{GL}_p(K)$-module
\[
F_p(R,L)=K\langle X_p\rangle/(K\langle X_p\rangle\cap I(R,L))\cong\sum_{\lambda}m_{\lambda}W_p(\lambda),
\]
where $K\langle X_p\rangle=K\langle x_1,\ldots,x_p\rangle$, the general linear group $\text{GL}_p(K)$ acts canonically
on the vector space $KX_p$ with basis $X_p=\{x_1,\ldots,x_p\}$
and this action is extended diagonally on the whole algebra $K\langle X_p\rangle$.

Our Theorem~\ref{main theorem} gives  explicitly the cocharacter sequence $\chi_n(M_3(K),so_3(K))$, $n=0,1,2,\ldots$.
The proof is based on a combination of classical invariant theory and representation theory of the general linear group.
By standard arguments due to Regev \cite{Regev},
since $\dim(so_3(K))=3$, we work in the algebra $F_3(M_3(K),so_3(K))$ considered as a $\text{GL}_3(K)$-module
instead to work with $P_n(M_3(K),so_3(K))$ and representations of $S_n$.
Using classical results from invariant theory we give  upper bounds for the multiplicities $m_{\lambda}$,
$\lambda=(\lambda_1,\lambda_2,\lambda_3)$ depending on the parity of
the differences $\lambda_1-\lambda_2$, $\lambda_2-\lambda_3$.
Then with explicit constructions we show that these upper bounds are achieved.

\subsection{Invariant theory of $\text{\rm SO}_3(K)$}
\label{subsec:invariant theory of SO_3}

The general linear group $\text{GL}_d(K)$ acts on the space
$M_d(K)^{\oplus p}$ of
$p$-tuples  of $d\times d$ matrices by simultaneous conjugation:
\[
g\cdot (r_1,\ldots,r_p)=(gr_1g^{-1},\ldots,gr_pg^{-1}),\quad g\in \text{GL}_d(K),\quad r_1,\ldots,r_p\in M_d(K).
\]
The polynomial algebra corresponding to this action is in $pd^2$ variables,
\[
K[Z_p]=K[z_{ij}^{(k)}\mid 1\leq i,j\leq d;\quad k=1,\ldots,p].
\]
The action of $\text{GL}_d(K)$ is defined in terms of generic $d\times d$ matrices
\[
z_k=\left(\begin{matrix}
z_{11}^{(k)}&\ldots&z_{1d}^{(k)}\\
\vdots&\ddots&\vdots\\
z_{d1}^{(k)}&\ldots&z_{dd}^{(k)}\\
\end{matrix}\right),\quad k=1,\ldots,p.
\]
If
\[g^{-1}\left(z_{ij}^{(k)}\right)g=\left(w_{ij}^{(k)}\right),\quad g\in \text{GL}_d(K),\quad  k=1,\ldots,p,
\]
then under the action of $g$ the variable $z_{ij}^{(k)}$ goes to $w_{ij}^{(k)}$.

The algebra of invariants of the orthogonal group $\text{O}_d(K)\subset \text{GL}_d(K)$ is described by Sibirskii \cite{sibirskii} and Procesi
\cite[Theorem 7.1]{Procesi:1}:

\begin{theorem}\label{orthogonal trace invariants}
The algebra $K[Z_p]^{\text{\rm O}_d(K)}$ of invariants of the group $\text{\rm O}_d(K)$
acting by simultaneous conjugation on $p$ copies of $M_d(K)$ is generated by the traces
\[
\text{\rm tr}(u_{k_1}\cdots u_{k_n}),\quad 1\leq k_1,\ldots,k_n\leq p,
\]
where $u_{k_r}=z_{k_r}$ or $u_{k_r}=z_{k_r}'$, the transpose of $z_{k_r}$, $r=1,\ldots,n$.
\end{theorem}

The generators of the algebra $K[Z_m]^{\text{SO}_d(K)}$ of invariants of $\text{SO}_d(K)$
are given by Aslaksen, Tan, and Zhu \cite[Theorem 3]{Aslaksen-Tan-Zhu}.

\begin{theorem}\label{special orthogonal trace invariants}
{\rm (i)} For $d$ odd the algebra $K[Z_p]^{\text{\rm SO}_d(K)}$ of $\text{\rm SO}_d(K)$-invariants
coincides with the algebra $K[Z_p]^{\text{\rm O}_d(K)}$ of $\text{\rm O}_d(K)$-invariants.

{\rm (ii)} For $d$ even $K[Z_p]^{\text{\rm SO}_d(K)}$ is generated by the generators of $K[Z_p]^{\text{\rm O}_d(K)}$
and the so called polarized Pfaffians.
\end{theorem}

The well-known generating system of the algebra of invariants $K[T_p]^{\text{SO}_3(K)}$ of the special orthogonal group $\text{SO}_3(K)$ acting
by simultaneous conjugation on $p$ copies of the Lie algebra $so_3(K)$ of $3\times 3$ skew-symmetric matrices
can be obtained as a consequence of  the special case $d=3$ of Theorems \ref{orthogonal trace invariants} and \ref{special orthogonal trace invariants}. 
For the rest of this section we assume $d=3$.  

\begin{corollary}\label{skew-symmetric invariants}
The algebra $K[T_p]^{SO_3(K)}$ is generated by the traces
\[
\text{\rm tr}(t_{k_1}\cdots t_{k_n}),\quad 1\leq k_1,\ldots,k_n\leq p.
\]
\end{corollary}

\begin{proof}
Since $so_3(K)^{\oplus p}$ is an $\mathrm{SO}_3(K)$-module direct summand of 
$M_3(K)^{\oplus p}$,  the substitution $z_k\mapsto t_k$, $k=1,\dots,p$ induces a $K$-algebra surjection $K[Z_p]^{\mathrm{SO}_3(K)}\to K[T_p]^{\mathrm{SO}_3(K)}$. Since $z_k'$ is mapped to $-t_k$, a generator 
$\text{\rm tr}(u_{k_1}\cdots u_{k_n})$ from Theorem~\ref{orthogonal trace invariants} is mapped to $\pm\text{\rm tr}(t_{k_1}\cdots t_{k_n})$. 
\end{proof}

In the sequel we shall refer to the algebra $K[T_p]^{SO_3(K)}$ generated by traces of products
of $3\times 3$ generic skew-symmetric matrices as the generic trace algebra.

\begin{remark} 
The algebra of $\mathrm{SL}_2(K)$-invariants under the adjoint action on $sl_2(K)^{\oplus p}$ is generated by traces of  monomials in the $2\times 2$ matrix components. 
\end{remark} 

Consider the space  $so_3(K)^{\oplus p}\oplus M_3(K)$, on which $\mathrm{SO}_3(K)$ acts by simultaneous conjugation,
and ${\mathrm{GL}}_p(K)$ acts on the right by
\[
(a_1,\dots,a_p,b)\cdot g=(\sum_{i=1}^3g_{i1}a_i,\dots,\sum_{i=1}^3g_{ip}a_i,b)
\]
for $g=(g_{ij})_{i,j=1}^p\in {\mathrm{GL}}_p(K)$.
The coordinate ring of $so_3(K)^{\oplus p}\oplus M_3(K)$ is $K[T_p,Z]$, where
$Z=\{z_{ij}\mid 1\le i,j\le 3\}$ is a set of  commuting indeterminates over $K[T_p]$.
We have the $K$-linear embedding
\begin{equation}\label{eq:embedding}
M_3(K[T_p])\to K[T_p,Z]^{(\mathbb{N}_0,\dots,\mathbb{N}_0,1)}
\text{ given by }f\mapsto \mathrm{tr}(fz),
\end{equation}
where $z=(z_{ij})_{i,j=1}^3$ is a generic $3\times 3$ matrix as in
Section~\ref{subsec:invariant theory of SO_3}, and $(\mathbb{N}_0,\dots,\mathbb{N}_0,1)$
in the exponent means that we take the component of $K[T_p,Z]$ consisting of the polynomial functions
that are linear on the summand $M_3(K)$ of $so_3(K)^{\oplus p}\oplus M_3(K)$.

\begin{proposition}\label{prop:embedding}
%\begin{itemize} \item[(i)]
{\rm (i)} The $K$-subalgebra ${\mathcal{E}}_p$ of $M_3(K[T_p])$ is generated
by the generic skew-symmetric matrices $t_1,\dots,t_p$ and the scalar matrices 
$\mathrm{tr}(t_{k_1}\cdots t_{k_n})I$ ($n\ge 2$, $1\le k_1,\dots,k_n\le p$).  

%\item[(ii)]
{\rm (ii)} The map $f\mapsto \mathrm{tr}(fz)$ gives a ${\mathrm{GL}}_p(K)$-module isomorphism
\[
\iota:{\mathcal{E}}_p\stackrel{\cong}{\longrightarrow} (K[T_p,Z]^{\mathrm{SO}_3(K)})^{(\mathbb{N}_0,\dots,\mathbb{N}_0,1)}.
\]
%\end{itemize}
\end{proposition}
\begin{proof}
By standard properties of the trace, the restriction of the embedding \eqref{eq:embedding}
maps ${\mathcal{E}}_p$ into $(K[T_p,Z]^{\mathrm{SO}_3(K)})^{(\mathbb{N}_0,\dots,\mathbb{N}_0,1)}$.
On the other hand, ${\mathcal{E}}_p$ contains the $K$-subalgebra generated by $t_1,\dots,t_p$, $\mathrm{tr}(t_{k_1}\dots t_{k_n})I$ $(1\le k_1,\dots,k_n\le p$), 
and the images of the elements of this subalgebra already exhaust
$(K[T_p,Z]^{\mathrm{SO}_3(K)})^{(\mathbb{N}_0,\dots,\mathbb{N}_0,1)}$ (and hence 
both (i) and (ii) hold): indeed, 
similarly to the proof of Corollary~\ref{skew-symmetric invariants}, the specialization 
$z_k\mapsto t_k$ ($k=1,\dots,p$), $z_{p+1}\mapsto z$ maps the generators of 
$K[Z_{p+1}]^{\mathrm{SO}_3(K)}$ given in Theorem~\ref{orthogonal trace invariants} to generators of $K[T_p,Z]^{\mathrm{SO}_3(K)}$.  
If such a generator is linear in $z$, then up to sign, it is of the form 
$\mathrm{tr}(t_{k_1}\cdots t_{k_n}z)$,  since  a matrix and its transpose have equal trace, therefore 
$\mathrm{tr}(t_{k_1}\cdots t_{k_n}z')=\mathrm{tr}(z''t_{k_n}'\cdots t_{k_1}')=(-1)^n\mathrm{tr}(t_{k_n}\cdots t_{k_1}z)$.  
Taking into account Corollary~\ref{skew-symmetric invariants} we conclude that the 
above subalgebra of ${\mathcal{E}}_p$ is mapped by $\iota$ onto 
$(K[T_p,Z]^{\mathrm{SO}_3(K)})^{(\mathbb{N}_0,\dots,\mathbb{N}_0,1)}$. 
\end{proof}

\begin{corollary} \label{cor:weyl}
For $p\ge 3$ we have ${\mathcal{E}}_p=\langle {\mathcal{E}}_3\rangle_{{\mathrm{GL}}_p(K)}$.
\end{corollary}

\begin{proof}
Since $\dim_K(so_3(K))=3$, by Weyl's Theorem on polarizations (derived from Capelli's identities in  \cite{Weyl}) we have
$K[T_p,Z]=\langle K[T_3,Z]\rangle_{{\mathrm{GL}}_p(K)}$, hence
\[
(K[T_p,Z]^{\mathrm{SO}_3(K)})^{(\mathbb{N}_0,\dots,\mathbb{N}_0,1)}=
\langle (K[T_3,Z]^{\mathrm{SO}_3(K)})^{(\mathbb{N}_0,\mathbb{N}_0,\mathbb{N}_0,1)}
\rangle_{{\mathrm{GL}}_p(K)}.
\]
So the statement follows by the isomorphism $\iota$ in Proposition~\ref{prop:embedding}.
\end{proof}

We need some facts on $K[T_p]^{\mathrm{SO}_3(K)}$ contained in the theorems  on  the vector invariants of the special orthogonal group, that we recall now. 
Let $V_d$ be the $d$-dimensional $K$-vector space with basis $\{v_1,\ldots,v_d\}$ with the canonical action
of the group $\text{GL}(V_d)$ identified in the usual way with $\text{GL}_d(K)$. The action of $\text{GL}_d(K)$ on $V_d$ induces an action
on the algebra $K[X_d]=K[x_1,\ldots,x_d]$ of polynomial functions on $V_d$
(here $x_1,\dots,x_d$ is the dual basis in $V_d^*$ to the basis chosen in $V_d$).
If
\[
v=\alpha_1v_1+\cdots+\alpha_dv_d\in V_d,\quad f=f(X_d)\in K[X_d], \quad g\in \text{GL}_d(K),
\]
then
\[
f(v)=f(\alpha_1,\ldots,\alpha_d) \text{ and }(g(f))(v)=f(g^{-1}(v)).
\]
For any subgroup $G$ of $\text{GL}_d(K)$ the algebra $K[X_d]^G$ of $G$-invariants
consists of all $f(X_d)\in K[X_d]$ with the property $g(f)=f$ for all $g\in G$.

We equip the vector space $V_d$ with a nondegenerate symmetric bilinear form. If
\[
v'=\alpha_1v_1+\cdots+\alpha_dv_d,\quad v''=\beta_1v_1+\cdots+\beta_dv_d,
\]
then
\[
\langle v',v''\rangle=\alpha_1\beta_1+\cdots+\alpha_d\beta_d.
\]
The special orthogonal group $\text{SO}_d(K)$ acts canonically on the vector space $V_d$ and
consists of all matrices with determinant equal to 1 which preserve the symmetric bilinear form.
The action of $\text{SO}_d(K)$ can be extended to the direct sum $V_d^{\oplus p}$ of $p$ copies of $V_d$.
Write $\{v_{i1},\ldots,v_{id}\}$ for the basis of the $i$th  direct summand of $V_d^{\oplus p}$
corresponding to the chosen basis $\{v_1,\dots,v_n\}$
of $V_d$,
and let $y_{ik}$ be the polynomial (in fact linear) function
which sends the vector $v_{ik}$ to 1 and to 0 all other vectors of the fixed basis of $V_d^{\oplus p}$.
Set $y_i=(y_{i1},\ldots,y_{id})$, $i=1,\ldots,p$,
and consider the scalar products
\[
\langle y_i,y_j\rangle=y_{i1}y_{j1}+\cdots+y_{id}y_{jd},\quad 1\leq i,j\leq p,
\]
the determinant
\[
\Delta_d(y_{j_1},\ldots,y_{j_d})=\det(y_{j_1},\ldots,y_{j_d})=
\left\vert\begin{matrix}
y_{1j_1}&y_{1j_2}&\ldots&y_{1j_d}\\
y_{2j_1}&y_{2j_2}&\ldots&y_{2j_d}\\
\vdots&\vdots&\ddots&\vdots\\
y_{dj_1}&y_{dj_2}&\ldots&y_{dj_d}\\
\end{matrix}\right\vert,
\]
$1\leq j_1<\cdots<j_d\leq p$, and the Gram determinant
\[
\Gamma_k(y_{i_1},\ldots,y_{i_k}\mid y_{j_1},\ldots,y_{j_k})=
\det(\langle y_{i_r},y_{j_s}\rangle)=
\left\vert\begin{matrix}
\langle y_{i_1},y_{j_1}\rangle&\ldots&\langle y_{i_1},y_{j_k}\rangle\\
\vdots&\ddots&\vdots\\
\langle y_{i_k},y_{j_1}\rangle&\ldots&\langle y_{i_k},y_{j_k}\rangle\\
\end{matrix}\right\vert,
\]
$1\leq i_1<\cdots<i_k\leq p$, $1\leq j_1<\cdots<j_k\leq p$.

The following classical theorems, see, e.g., \cite[Theorems 2.9.A and 2.17.A]{Weyl},
describe the generating set and the defining relations of the algebra
\[
K[Y_{pd}]^{\text{SO}_d(K)}=K[y_{ik}\mid i=1,\ldots,p; \ k=1,\ldots,d]^{\text{SO}_d(K)}
\]
of $\text{SO}_d(K)$-invariants of $V_d^{\oplus p}$.

\begin{theorem}[First fundamental theorem for the invariants of $\text{\rm SO}_d(K)$]\label{first fundamental theorem}
{\rm (i)} The algebra $K[Y_{pd}]^{\text{\rm SO}_d(K)}$ is generated by the scalar products
$\langle y_i,y_j\rangle$, $1\leq i,j\leq p$,
and by the determinants $\Delta_d(y_{j_1},\ldots,y_{j_d})$, $1\leq j_1<\cdots<j_d\leq p$.

{\rm (ii)} The elements of $K[Y_{pd}]^{\text{\rm SO}_d(K)}$ are linear combinations of products
\[
\langle y_{i_1},y_{j_1}\rangle\cdots\langle y_{i_n},y_{j_n}\rangle
\text{ and }\Delta_d(y_{k_1},\ldots,y_{k_d})\langle y_{i_1},y_{j_1}\rangle\cdots\langle y_{i_n},y_{j_n}\rangle,
\]
\[
1\leq i_r,j_r\leq p, \quad r=1,\ldots,n,\quad 1\leq k_1<\cdots<k_d\leq p.
\]
\end{theorem}

\begin{theorem}[Second fundamental theorem for the invariants of $\text{\rm SO}_d(K)$]\label{second fundamental theorem}
The defining relations of the algebra $K[Y_{pd}]^{\text{\rm SO}_d(K)}$ consist of
\[
\Gamma_{d+1}(y_{i_0},y_{i_1},\ldots,y_{i_d}\mid y_{j_0},y_{j_1},\ldots,y_{j_d})=0,
\]
\[
1\leq i_0<i_1<\cdots <i_d\leq p,\quad 1\leq j_0<j_1<\cdots <j_d\leq p,
\]
\[
\Delta_d(y_{i_1},\ldots,y_{i_d})\Delta_d(y_{j_1},\ldots,y_{j_d})-
\Gamma_d(y_{i_1},\ldots,y_{i_d}\mid y_{j_1},\ldots,y_{j_d})=0,
\]
\[
1\leq i_1<\cdots <i_d\leq p,\quad 1\leq j_1<\cdots <j_d\leq p,
\]
\[
\sum_{r=0}^d(-1)^r\langle y_i,y_{j_r}\rangle
\Delta_d(y_{j_0},\ldots,\hat y_{j_r},\ldots,y_{j_d})=0,
\]
\[
1\leq i\leq p,\quad 1\leq j_0<j_1<\cdots<j_d\leq p,
\]
where $\hat y_{j_r}$ means that $y_{j_r}$ does not participate in the expression.
\end{theorem}

In \cite{Domokos-Drensky} we found a Gr\"obner basis of the ideal of defining relations
of the algebra $K[Y_{pd}]^{\text{\rm SO}_d(K)}$.

Let ${\mathbb N}_0=\{0,1,2,\ldots\}$. We define an ${\mathbb N}_0^p$-grading on the polynomial algebras $K[Y_{p3}]$, $K[T_p]$ 
and on the algebra ${\mathcal{F}}_p$ assuming that the variables $y_{kj}$, $t_{ij}^{(k)}$ and the matrix $t_k$ are of degree $(0,\ldots,0,1,0,\ldots,0)$
(the $k$th coordinate is equal to 1 and all other coordinates are equal to 0). The generic trace algebra is an ${\mathbb N}_0^p$-graded subalgebra of $K[T_p]$. 

\begin{proposition}\label{isomorphic algebras of invariants}
The algebras $K[Y_{p3}]^{\text{\rm SO}_3(K)}$ and $K[T_p]^{\text{\rm SO}_3(K)}$
are isomorphic as ${\mathbb N}_0^p$-graded algebras.
\end{proposition}

\begin{proof}
The vector space $so_3(K)$ has a basis
\[
a_1=\left(\begin{matrix}
0&1&0\\
-1&0&0\\
0&0&0\\
\end{matrix}\right),\quad
a_2=\left(\begin{matrix}
0&0&1\\
0&0&0\\
-1&0&0\\
\end{matrix}\right),\quad
a_3=\left(\begin{matrix}
0&0&0\\
0&0&1\\
0&-1&0\\
\end{matrix}\right).
\]
Denoting by $e_1,e_2,e_3$ the standard basis vectors in the space $K^3$ of column vectors, a straightforward calculation shows that
the linear map
\[
so_3(K)\to K^3,\quad a_1\mapsto e_3,\ a_2\mapsto -e_2,\ a_3\mapsto e_1
\]
is an isomorphism between the $\text{SO}_3(K)$-modules $so_3(K)$ and $K^3$,
where $\text{SO}_3(K)$ acts via conjugation on $so_3(K)$ and via matrix multiplication on $K^3$. This isomorphism induces an isomorphism of
the $\text{SO}_3(K)$-modules $so_3(K)^{\oplus p}\cong (K^3)^{\oplus p}$, their coordinate rings $K[T_p]\cong K[Y_{p3}]$, and finally the
${\mathbb N}_0^p$-graded subalgebras $K[T_p]^{\text{SO}_3(K)}\cong K[Y_{p3}]^{\text{SO}_d(K)}$ of $\text{SO}_3(K)$-invariants.
For sake of completeness of the picture, we mention that
the  basis $\{a_1,a_2,a_3\}$ in $so_3(K)$ is  orthonormal
with respect to the nondegenerate, symmetric, $\text{SO}_3(K)$-invariant bilinear form defined by
\[
\langle a,b\rangle=-\frac 12\text{tr}(ab),\quad a,b\in so_3(K).
\]
\end{proof}

For a skew-symmetric $3\times 3$ matrix $a$ and a symmetric $3\times 3$ matrix $b$
we have $\mathrm{tr}(ab)=0$. It follows that $\mathrm{tr}(t_1t_2t_3)+\mathrm{tr}(t_2t_1t_3)=\mathrm{tr}((t_1t_2+t_2t_1)t_3)=0$, so for any permutation $\pi\in S_3$ we have
\[
\mathrm{tr}(t_{\pi(1)}t_{\pi(2)}t_{\pi(3)})=\mathrm{sign}(\pi)\mathrm{tr}(t_1t_2t_3).
\]

\begin{corollary}\label{cor:tr(xyz)}
%\begin{itemize}
%\item[(i)]
{\rm (i)} The algebra $K[T_p]^{\mathrm{SO}_3(K)}$ is generated by the elements $\mathrm{tr}(t_it_j)$, $1\le i\le j\le p$, and  $\mathrm{tr}(t_kt_lt_m)$, $1\le k<l<m\le p$.
%\item[(ii)]

{\rm (ii)}
The algebra $K[T_3]^{\mathrm{SO}_3(K)}$ is a rank two free module
generated by $\mathrm{tr}(t_1t_2t_3)$ over its subalgebra generated by the algebraically independent elements
$\mathrm{tr}(t_1^2)$, $\mathrm{tr}(t_2^2)$, $\mathrm{tr}(t_3^2)$, $\mathrm{tr}(t_1t_2)$, $\mathrm{tr}(t_1t_3)$, $\mathrm{tr}(t_2t_3)$.
%\end{itemize}
\end{corollary}

\begin{proof} (i) is an immediate consequence of Theorem~\ref{first fundamental theorem},
Corollary~\ref{skew-symmetric invariants}, Proposition~\ref{isomorphic algebras of invariants}.
Taking into account also Theorem~\ref{second fundamental theorem}, we get (ii).
\end{proof}

\begin{corollary}\label{cor:cov-generators} 
The $K$-subalgebra ${\mathcal{E}}_p$ of $M_3(K[T_p])$ is generated
by the generic skew-symmetric matrices $t_1,\dots,t_p$ and the scalar matrices $\mathrm{tr}(t_it_j)I$ ($1\le i\le j\le p$), $\mathrm{tr}(t_kt_lt_m)I$ ($1\le k<l<m\le p$).
\end{corollary} 
\begin{proof} 
This follows from Proposition~\ref{prop:embedding} (i) and Corollary~\ref{cor:tr(xyz)} (i). 
\end{proof} 

\subsection{Representation theory of ${\mathrm{GL}}_p(K)$}
\label{subsec:rep-GL}

In what follows we assume that the general linear group $\text{GL}_p(K)=\text{GL}(KX_p)$ acts
canonically on the vector space $KX_p$ with basis $X_p$. That is, for 
\[g=(g_{ij})_{i,j=1}^p\in \mathrm{GL}_p(K) \text{ we have }g(x_j)=\sum_{i=1}^pg_{ij}x_i,\quad j=1,\ldots,p.
\]
This action can be extended diagonally on the tensor algebra
\[
T(KX_p)=\sum_{n\geq 0}(KX_p)^{\otimes n}\cong K\langle X_p\rangle.
\]
In the sequel we shall identify $T(KX_p)$ with the free associative algebra $K\langle X_p\rangle$
and $(KX_p)^{\otimes n}$ with the homogeneous component $K\langle X_p\rangle^{(n)}$ of degree $n$ of
$K\langle X_p\rangle$. 
The standard $\mathbb{N}_0^p$-grading on $K\langle X_p\rangle$ corresponds to the decomposition of $K\langle X_p\rangle$ into the direct sum of the isotypic components under the action of the subgroup of diagonal matrices in $\mathrm{GL}_p(K)$. 
The $\text{GL}_p(K)$-module $K\langle X_p\rangle$ is a direct sum of irreducible polynomial $\text{GL}_p(K)$-modules.
The irreducible polynomial $\text{GL}_p(K)$-modules are indexed by partitions having  not more than $p$ parts
and all they appear as summands in $K\langle X_p\rangle$. Let
\[
\lambda=(\lambda_1,\ldots,\lambda_p),\quad \lambda_1\geq\cdots\geq\lambda_p\geq 0,
\quad\lambda_1+\cdots+\lambda_p=n,
\]
be a partition of $n$ and let $W_p(\lambda)$ be the corresponding $\text{GL}_p(K)$-module.
By Schur-Weyl duality (cf. \cite{Weyl} or \cite[page 256, (3.1.4)]{Procesi:3}), 
the homogeneous component $K\langle X_p\rangle^{(n)}$ of $K\langle X_p\rangle$ decomposes as
\[
K\langle X_p\rangle^{(n)}\cong \sum_{\lambda\vdash n}\deg(\lambda)W_p(\lambda),
\]
where $\deg(\lambda)$ is the degree of the irreducible $S_n$-character $\chi_{\lambda}$.
A non-zero element of $K\langle X_p\rangle^{(n)}$ is called a 
\emph{highest weight vector of weight $\lambda$} if it is 
fixed by the subgroup of unipotent upper triangular matrices in $\mathrm{GL}_p(K)$, and it  is multihomogeneous of ${\mathbb N}_0^p$-degree $\lambda$. 
Any $\text{GL}_p(K)$-submodule $W\subset K\langle X_p\rangle^{(n)}$, $W\cong W_p(\lambda)$ 
contains a unique (up to non-zero scalar multiples) highest weight vector (necessarily having weight $\lambda$ and generating $W$ as a $\mathrm{GL}_p(K)$-module). The 
highest weight vectors in  $K\langle X_p\rangle^{(n)}$ can be described in the following way. The symmetric group $S_n$ acts from the right  on $K\langle X_p\rangle^{(n)}$
by the rule
\[
(x_{i_1}\cdots x_{i_n})^{\tau}=x_{i_{\tau(1)}}\cdots x_{i_{\tau(n)}},\quad \tau\in S_n,
\]
and this action commutes with the action of $\text{GL}_p(K)$ introduced before.
Let $[\lambda]$ be the Young diagram corresponding to the partition $\lambda$
and let the lengths of the columns of $[\lambda]$ be $k_1,\ldots,k_{\lambda_1}$.
Consider the product of standard polynomials
\[
w_{\lambda}(x_1,\ldots,x_{k_1})=\prod_{j=1}^{\lambda_1}s_{k_j}(x_1,\ldots,x_{k_j})
=\prod_{j=1}^{\lambda_1}\left(\sum_{\sigma_j\in S_{k_j}}\text{sign}(\sigma_j)x_{\sigma_j(1)}\cdots x_{\sigma_j(k_j)}\right).
\]
Then every highest weight vector of weight $\lambda$ is of the form
\[
w=\sum_{\tau\in S_n}\alpha_{\tau}w_{\lambda}^{\tau},\quad \alpha_{\tau}\in K.
\]
A $\lambda$-tableau is the Young diagram $[\lambda]$ whose boxes are filled with positive integers. We say that the tableau is of content $(n_1,\ldots,n_p)$
if $1,\ldots,p$ appear in it $n_1,\ldots,n_p$ times, respectively.
The tableau is standard
if its entries are the numbers $1,\ldots,n$, without repetition, arranged in such a way that they increase in rows (reading them from left to right)
and in columns (reading from top to bottom).
It is semistandard if its entries (allowing repetitions) do not decrease in rows and increase in columns.

Given a partition $\lambda$ of $n$, we set up a bijection between the set of $\lambda$-tableaux of content $(1,\dots,1)$ and $S_n$ as follows:
we assign to the permutation
$\varrho\in S_n$ the Young tableau $T_{\lambda}(\varrho)$ obtained by filling in the boxes
 of the first column of $[\lambda]$ with $\varrho^{-1}(1),\ldots,\varrho^{-1}(k_1)$, of the second column with
$\varrho^{-1}(k_1+1),\ldots,\varrho^{-1}(k_1+k_2)$, etc.
Then the highest weight vector $w_{\lambda}^{\varrho}$ has skew-symmetries in the positions listed in the first column of $T_{\lambda}(\varrho)$,
skew-symmetries in the positions listed in the second  column of $T_{\lambda}(\varrho)$, etc.
For example, for $n=5$, $\lambda=(2,2,1)$, and
\[
\varrho^{-1}=\left(\begin{matrix}1&2&3&4&5\\
1&5&3&2&4\\
\end{matrix}\right), \text{ we have }
\]
\vbox{$\displaystyle T_{\lambda}(\varrho)=\young(12,54,3), \quad w_{\lambda}^{\varrho}=\sum_{\sigma_1\in S_3,\sigma_2\in S_2}
\text{sign}(\sigma_1)\text{sign}(\sigma_2) x_{\sigma_1(1)}x_{\sigma_2(1)}x_{\sigma_1(3)}x_{\sigma_2(2)}x_{\sigma_1(2)}$.\\
\phantom{xxx}}
It is known (see e.g. \cite{Drensky:2}) that the set of all $w_{\lambda}^{\varrho}$ corresponding to the standard $\lambda$-tableaux $T_{\lambda}(\varrho)$ is a basis
of the vector space of the highest weight vectors of weight $\lambda$ in 
$K\langle X_p\rangle^{(n)}$. 
We note also that if $w_i$ $(i=1,2)$ is a highest weight vector of weight $\lambda^{(i)}\vdash n_i$ in $K\langle X_p\rangle^{(n_i)}$, then the product $w_1w_2$ is a highest weight vector of weight $\lambda^{(1)}+\lambda^{(2)}$ in    
$K\langle X_p\rangle^{(n_1+n_2)}$. 

\begin{proposition}\label{correspondence with semistandard tableaux}
%\begin{itemize}\item[(i)]
{\rm (i)} There is a one-to-one correspondence between an arbitrary ${\mathbb N}_0^p$-graded basis
of a  ${\mathrm{GL}}_p(K)$-submodule of $K\langle X_p\rangle$ isomorphic to $W_p(\lambda)$ 
and the set of semistandard $\lambda$-tableaux filled in with $1,\ldots,p$,
such that a basis vector of degree $(n_1,\ldots,n_p)$ corresponds to a semistandard tableau of content  $(n_1,\ldots,n_p)$.

%\item[(ii)]
{\rm (ii)} Let $W$ be a polynomial ${\mathrm{GL}}_p(K)$-module 
(i.e. $W$ is the direct sum of modules isomorphic to $W_p(\lambda)$ for various $\lambda$), 
endowed with the ${\mathbb N}_0^p$-grading given by the action of the subgroup of diagonal matrices in 
${\mathrm{GL}}_p(K)$. 
Suppose that there exists a mapping $\pi$ from an  ${\mathbb N}_0^p$-graded basis of $W$  into the set of semistandard $\lambda$-tableaux, such that
a basis vector of degree $(n_1,\ldots,n_p)$ is mapped to a semistandard tableau of content  $(n_1,\ldots,n_p)$,
and for each partition $\lambda$, there exists a non-negative integer
$m_{\lambda}$ such that that every semistandard $\lambda$-tableau
is the image of exactly $m_{\lambda}$ basis elements.  Then $W$ decomposes as
\[
W=\sum_{\lambda}m_{\lambda}W_p(\lambda).
\]
%\end{itemize}
\end{proposition}

\begin{proof}
The statement (i) follows immediately from the fact that the dimension of the homogeneous component $W_p^{(n_1,\ldots,n_p)}(\lambda)$
of degree $(n_1,\ldots,n_p)$ is equal to the coefficient of $\xi_1^{n_1}\cdots \xi_d^{n_p}$ of the Schur function $S_{\lambda}(\xi_1,\ldots,\xi_p)$.
On the other hand this coefficient is equal to the number of semistandard $\lambda$-tableaux of content $(n_1,\ldots,n_d)$.
For (ii) it is sufficient to apply the fact that the Schur function plays the role of character of the representation of $\text{GL}_p(K)$
corresponding to the $\text{GL}_p(K)$-module $W_p(\lambda)$ and that the character of the direct sum of polynomial representations determines
the decomposition of the corresponding $\text{GL}_p(K)$-module $W_p$.
\end{proof}

The decomposition of the $\text{GL}_p(K)$-module structure of the algebra of invariants of $\text{SO}_3(K)$ acting on $p$ copies of $V_3$
is given for example in \cite[Section 1.2]{Procesi:2} or in \cite[Chapter I, Theorem 4.3]{Le Bruyn} in terms of semistandard tableaux.

\begin{theorem}\label{theorem of Le Bruyn}
The algebra $K[Y_{p3}]^{\text{\rm SO}_3(K)}$
has an ${\mathbb N}_0^p$-graded basis indexed
(via a mapping $\pi$ as in Proposition~\ref{correspondence with semistandard tableaux} (ii))
by all semistandard $\lambda$-tableaux
for all $\lambda=(2\mu_1,2\mu_2,2\mu_3)$ and $\lambda=(2\mu_1+1,2\mu_2+1,2\mu_3+1)$, where $\mu_1,\mu_2,\mu_3\in \mathbb{N}_0$. 
\end{theorem}

As an immediate consequence of Propositions \ref{isomorphic algebras of invariants} and
\ref{correspondence with semistandard tableaux} and Theorem \ref{theorem of Le Bruyn} we obtain:

\begin{corollary}\label{decomposition of special orthogonal invariants}
As a ${\mathrm{GL}}_p(K)$-module the algebra $K[T_p]^{\text{\rm SO}_3(K)}$ of invariants of
the action by simultaneous conjugation of $\text{\rm SO}_3(K)$ on $p$ copies of $3\times 3$ skew-symmetric matrices decomposes as
\[
K[T_p]^{\text{\rm SO}_3(K)}\cong \sum W_p(2\mu_1+\delta,2\mu_2+\delta,2\mu_3+\delta),
\]
where the summation runs on all partitions $(\mu_1,\mu_2,\mu_3)$ and $\delta=0$ or $1$.
\end{corollary}

\section{The cocharacter sequence and  highest weight vectors}\label{sec:main result}

Since $\dim(so_3(K))=3$, by a theorem of Regev \cite{Regev} 
the cocharacter sequence of $I(M_3(K),so_3(K))$
is of the form
\[
\chi_n(M_3(K),so_3(K))=\sum_{\lambda\vdash n}m_{\lambda}\chi_{\lambda},
\]
where $\lambda=(\lambda_1,\lambda_2,\lambda_3)$ is a partition of $n$ in not more than three parts.
This allows to replace the problem for the cocharacter sequence with the problem of the decomposition
into a direct sum of irreducible components of the $\text{GL}_3(K)$-module $F_3(M_3(K),so_3(K))$. 
The map $x_i\mapsto t_i$ induces an isomorphism 
\begin{equation}\label{eq:F3F3} 
F_3(M_3(K),so_3(K))\cong {\mathcal{F}}_3.\end{equation}   
The algebra ${\mathcal{F}}_3$ is contained in ${\mathcal{E}}_3$. 
By Proposition~\ref{prop:embedding} we have 
\begin{equation}\label{eq:F3E3}  
{\mathcal{F}}_3\subseteq {\mathcal{E}}_3\cong  (K[T_3,Z]^{\mathrm{SO}_3(K)})^{(\mathbb{N}_0,\dots,\mathbb{N}_0,1)}.\end{equation} 
Thus the multiplicities of 
the irreducible $\text{GL}_3(K)$-modules in ${\mathcal{F}}_3$ are bounded by their multiplicities in 
${\mathcal{E}}_3$, and the latter can be computed using 
Corollary~\ref{decomposition of special orthogonal invariants}, thanks to Lemma~\ref{lemma:a4a5} below. 

To state Lemma~\ref{lemma:a4a5} we need some notation. 
Let $(K[T_5]^{\text{SO}_3(K)})^{({\mathbb N}_0,{\mathbb N}_0,{\mathbb N}_0,1,1)}$
be the component of the generic trace algebra $K[T_5]^{\text{SO}_3(K)}$ which is linear in the generic skew-symmetric matrices
$t_4$ and $t_5$. Embedding $\text{GL}_3(K)$ into $\text{GL}_5(K)$ by
\[
\text{GL}_3(K)\ni g=\left(\begin{matrix}
g_{11}&g_{12}&g_{13}\\
g_{21}&g_{22}&g_{23}\\
g_{31}&g_{32}&g_{33}\\
\end{matrix}\right)\mapsto
\left(\begin{matrix}
g_{11}&g_{12}&g_{13}&0&0\\
g_{21}&g_{22}&g_{23}&0&0\\
g_{31}&g_{32}&g_{33}&0&0\\
0&0&0&1&0\\
0&0&0&0&1\\
\end{matrix}\right)\in \text{GL}_5(K)
\]
we equip the vector space $(K[T_5]^{\text{SO}_3(K)})^{({\mathbb N}_0,{\mathbb N}_0,{\mathbb N}_0,1,1)}$
with the structure of a $\text{GL}_3(K)$-module.

\begin{lemma}\label{lemma:a4a5} 
The comorphism of the map 
\begin{align*}
\mu:so_3(K)^{\oplus 5} \to so_3(K)^{\oplus 3}\oplus M_3(K),
\\ (a_1,a_2,a_3,a_4,a_5)\mapsto (a_1,a_2,a_3,a_4\cdot a_5)
\end{align*}
gives a ${\mathrm{GL}}_3(K)$-module isomorphism
\[
\mu^*:(K[T_3,Z]^{\mathrm{SO}_3(K)})^{(\mathbb{N}_0,\mathbb{N}_0,\mathbb{N}_0,1)}
\stackrel{\cong}\longrightarrow (K[T_5])^{\mathrm{SO}_3(K)})^{(\mathbb{N}_0,\mathbb{N}_0,\mathbb{N}_0,1,1)}.
\]
\end{lemma} 
\begin{proof} 
Observe that $so_3(K)\oplus so_3(K)\to M_3(K)$, $(a,b)\mapsto ab$ is the algebraic quotient 
map for the action of the multiplicative group $K^\times$ on $so_3(K)\oplus so_3(K)$ given by $c\cdot (a,b)=(ca,c^{-1}b)$. Indeed, the algebra $K[T_2]^{K^\times}$ of $K^\times$-invariants 
on $so_3(K)\oplus so_3(K)$ is generated by all products $t_{ij}^{(1)}t_{kl}^{(2)}$, and these products span the same $K$-subspace in $K[T_2]$ as the entries of the product $t_1t_2$ 
of the generic skew-symmetric matrices $t_1$ and $t_2$.  
It follows that $\mu$ is the algebraic quotient map for the action 
of $K^\times$ on $so_3(K)^{\oplus 5}$ given by 
\[c\cdot (a_1,a_2,a_3,a_3,a_5)=(a_1,a_2,a_3,ca_4,c^{-1}a_5),\quad c\in K^\times.\]
Therefore the comorphism $\mu^*$ of $\mu$ maps $K[T_3,Z]$ onto 
\[\mu^*(K[T_3,Z])=K[T_5]^{K^\times}=
\bigoplus_{j=0}^{\infty}K[T_5]^{(\mathbb{N}_0,\mathbb{N}_0,\mathbb{N}_0,j,j)},\] 
where $(\mathbb{N}_0,\mathbb{N}_0,\mathbb{N}_0,j,j)$ in the exponent means that we take the sum of the multihomogeneous components with multidegree 
$(\alpha_1,\alpha_2,\alpha_3,j,j)$, $\alpha_1,\alpha_2,\alpha_3$ ranging over $\mathbb{N}_0$. As the action of $K^\times$ commutes with the action of $\mathrm{SO}_3(K)$ on 
$so_3^{\oplus 5}$, the comorphism $\mu^*$ is $\mathrm{SO}_3(K)$-equivariant, and we have 
\[\mu^*(K[T_3,Z]^{\mathrm{SO}_3(K)})^{(\mathbb{N}_0,\mathbb{N}_0,\mathbb{N}_0,j)})
=(K[T_5]^{\mathrm{SO}_3(K)})^{(\mathbb{N}_0,\mathbb{N}_0,\mathbb{N}_0,j,j)}, 
\quad j=0,1,2,\dots.\] 
The restriction of $\mu^*$ to  
$K[T_3,Z]^{(\mathbb{N}_0,\mathbb{N}_0,\mathbb{N}_0,1)}$ is injective, because the image of the multiplication map $so_3(K)\oplus so_3(K)\to M_3(K)$ 
spans $M_3(K)$ as a $K$-vectorspace, hence a linear function on $M_3(K)$ vanishing on all $\{ab\mid a,b\in so_3(K)\}$ must be the zero map. 
Thus the restriction of $\mu^*$ to $K[T_3,Z]^{\mathrm{SO}_3(K)})^{(\mathbb{N}_0,\mathbb{N}_0,\mathbb{N}_0,1)}$ is a vector space isomorphism onto 
$(K[T_5])^{\mathrm{SO}_3(K)})^{(\mathbb{N}_0,\mathbb{N}_0,\mathbb{N}_0,1,1)}$. 
Moreover, it is a $\mathrm{GL}_3(K)$-module homomorphism, because $\mu$ is obviously 
a $\mathrm{GL}_3(K)$-module homomorphism, and the actions of $\mathrm{SO}_3(K)$ and $K^\times$ both commute with the action of $\mathrm{GL}_3(K)$. 
\end{proof} 

\begin{lemma}\label{semistandard tableaux in three rows}
Let $\lambda=(\lambda_1,\lambda_2,\lambda_3)=(2\mu_1+\delta,2\mu_2+\delta,2\mu_3+\delta)$, $\delta=0$ or $1$.
Consider the set of semistandard
$\lambda$-tableaux
of content $(n_1,n_2,n_3,1,1)$.
Deleting the boxes containing $4$ and $5$ from each tableau, we obtain a multiset
of semistandard tableaux of content $(n_1,n_2,n_3)$.
%\begin{itemize}
%\item[(i)]
{\rm (i)} The multiplicity of a semistandard $\nu$-tableau of content $(n_1,n_2,n_3)$ in this multiset
is non-zero if and only if $\nu=(\nu_1,\nu_2,\nu_3)$, $\nu_1\ge\nu_2\ge\nu_3\ge 0$,
and
\begin{align*}
\nu\in &\{
(\lambda_1-2,\lambda_2,\lambda_3),\ (\lambda_1,\lambda_2-2,\lambda_3),\ (\lambda_1,\lambda_2,\lambda_3-2),
\\ &
(\lambda_1-1,\lambda_2-1,\lambda_3),\ (\lambda_1-1,\lambda_2,\lambda_3-1),\ (\lambda_1,\lambda_2-1,\lambda_3-1)\}.
\end{align*}
%\item[(ii)]

{\rm (ii)} Moreover, the multiplicity is $2$ if 
$\nu=(\lambda_1-1,\lambda_2-1,\lambda_3)$ and $\lambda_1>\lambda_2$, 
or $\nu=(\lambda_1,\lambda_2-1,\lambda_3-1)$ and $\lambda_2>\lambda_3$, 
or $\nu=(\lambda_1-1,\lambda_2,\lambda_3-1)$. 

%\item[(iii)]
{\rm (iii)} All other positive multiplicities are equal to $1$.
%\end{itemize}
\end{lemma}

\begin{proof}
If $\lambda_1-\lambda_2,\lambda_2-\lambda_3,\lambda_3\geq 2$, then the semistandard $\lambda$-tableaux of content $(n_1,n_2,n_3,1,1)$
are of the following form:

\young(\:\:\:\:\:\:\:45,\:\:\:\:\:,\:\:\:)\hskip0.5truecm
\young(\:\:\:\:\:\:\:\:\:,\:\:\:45,\:\:\:)\hskip0.5truecm
\young(\:\:\:\:\:\:\:\:\:,\:\:\:\:\:,\:45)

\young(\:\:\:\:\:\:\:\:4,\:\:\:\:5,\:\:\:)\hskip0.5truecm
\young(\:\:\:\:\:\:\:\:4,\:\:\:\:\:,\:\:5)\hskip0.5truecm
\young(\:\:\:\:\:\:\:\:\:,\:\:\:\:4,\:\:5)

\young(\:\:\:\:\:\:\:\:5,\:\:\:\:4,\:\:\:)\hskip0.5truecm
\young(\:\:\:\:\:\:\:\:5,\:\:\:\:\:,\:\:4)\hskip0.5truecm
\young(\:\:\:\:\:\:\:\:\:,\:\:\:\:5,\:\:4)

\medskip

\noindent If $\lambda_1=\lambda_2>\lambda_3$ or $\lambda_2=\lambda_3>0$, then
4 and 5  may appear in the same column,  and 4 is necessarily above 5.
These observations clearly yield the statements (i), (ii), (iii).
\end{proof}

\begin{lemma}\label{semistandard tableaux from 5 to 3}
The ${\mathrm{GL}}_3(K)$-module $(K[T_5]^{\text{\rm SO}_3(K)})^{({\mathbb N}_0,{\mathbb N}_0,{\mathbb N}_0,1,1)}$
decomposes as
\[
(K[T_5]^{\text{\rm SO}_3(K)})^{({\mathbb N}_0,{\mathbb N}_0,{\mathbb N}_0,1,1)}
=\sum_{\nu}m_{\nu}W_3(\nu),\quad \nu=(\nu_1,\nu_2,\nu_3),
\]
where:

\noindent{\rm (i)} If $\nu_1\equiv\nu_2\equiv\nu_3\text{ \rm (mod 2)}$, then
\[
m_{\nu}=\begin{cases}3,&\text{if }\nu_1>\nu_2>\nu_3;\\
2,&\text{if }\nu_1=\nu_2>\nu_3;\\
2,&\text{if }\nu_1>\nu_2=\nu_3;\\
1,&\text{if }\nu_1=\nu_2=\nu_3;\\
\end{cases}
\]
{\rm (ii)} If $\nu_1\equiv\nu_2\not\equiv\nu_3\text{ \rm (mod 2)}$, then
\[
m_{\nu}=\begin{cases}2,&\text{if }\nu_1>\nu_2;\\
1,&\text{if }\nu_1=\nu_2;\\
\end{cases}
\]
{\rm (iii)} If $\nu_1\equiv\nu_3\not\equiv\nu_2\text{ \rm (mod 2)}$, then
$m_{\nu}=2$;

\noindent {\rm (iv)} If $\nu_1\not\equiv\nu_2\equiv\nu_3\text{ \rm (mod 2)}$, then
\[
m_{\nu}=\begin{cases}2,&\text{if }\nu_2>\nu_3;\\
1,&\text{if }\nu_2=\nu_3.\\
\end{cases}
\]
\end{lemma}

\begin{proof}
Combining Proposition~\ref{isomorphic algebras of invariants} and 
Theorem \ref{theorem of Le Bruyn} we obtain that as an ${\mathbb N}_0^5$-graded vector space the algebra $K[T_5]^{\text{SO}_3(K)}$
has a graded basis indexed by all semistandard $\lambda$-tableaux
for all $\lambda=(\lambda_1,\lambda_2,\lambda_3)$, such that $\lambda_1\equiv\lambda_2\equiv\lambda_3\text{ (mod 2)}$.
Hence the vector space $(K[T_5]^{\text{SO}_3(K)})^{({\mathbb N}_0,{\mathbb N}_0,{\mathbb N}_0,1,1)}$
has a basis indexed by the semistandard $\lambda$-tableaux of content $(n_1,n_2,n_3,1,1)$, $n_1,n_2,n_3\in{\mathbb N}_0$.
Deleting 4 and 5 from such a semistandard $\lambda$-tableau, we obtain the semistandard $\nu$-tableaux of content
$(n_1,n_2,n_3)$ described in Lemma \ref{semistandard tableaux in three rows}.

(i) If $\nu_1\equiv\nu_2\equiv\nu_3\text{ (mod 2)}$
and $\nu_1>\nu_2>\nu_3$ then we can obtain the $\nu$-tableau only from the corresponding $\lambda$-tableaux for
$\lambda=(\nu_1+2,\nu_2,\nu_3),(\nu_1,\nu_2+2,\nu_3),(\nu_1,\nu_2,\nu_3+2)$. 
By Proposition~\ref{correspondence with semistandard tableaux} (ii) we conclude   $m_{\nu}=3$. 
If $\nu_1=\nu_2>\nu_3$ we have to exclude the case $\lambda=(\nu_1,\nu_2+2,\nu_3)$ because $\nu_1<\nu_2+2$.
The other cases $\nu_1>\nu_2=\nu_3$ and $\nu_1=\nu_2=\nu_3$ are handled in a similar way.

(ii) If $\nu_1\equiv\nu_2\not\equiv\nu_3\text{ (mod 2)}$ and $\nu_1>\nu_2$, then we can obtain the $\nu$-tableau
from the two $\lambda$-tableaux for $\lambda=(\nu_1+1,\nu_2+1,\nu_3)$, i.e., $m_{\nu}=2$.
When $\nu_1=\nu_2$ there is only one $\lambda$-tableau $\lambda=(\nu_1+1,\nu_2+1,\nu_3)$ when 4 and 5 are in the most right column
of the $\lambda$-tableau.

The proofs of the other two cases (iii) and (iv) are similar.
\end{proof}

By \eqref{eq:F3F3}, \eqref{eq:F3E3} and Lemma~\ref{lemma:a4a5}, 
we get the following corollary: 

\begin{corollary}\label{bounds for cocharacters}
The multiplicities of the irreducible components $W_3(\nu)$ in the decomposition of
\[
F_3(M_3(K),so_3(K))\cong \sum_{\nu}m_{\nu}(M_3(K),so_3(K))W_3(\nu)
\]
are bounded from above by the integers $m_{\nu}$ in Lemma \ref{semistandard tableaux from 5 to 3}.
\end{corollary}

We turn to a construction of highest weight vectors in $K\langle X_3\rangle$ that are 
linearly independent modulo $I(M_3(K),so_3(K))$.  
As we shall see, for almost all partitions $\nu$, there exist as many of those 
as the upper bound $m_{\nu}$ in Corollary~\ref{bounds for cocharacters} for the multiplicity of $W_3(\nu)$ in $F_3(M_3(K),so_3(K))$.  
 
For a partition $\lambda=(\lambda_1,\lambda_2,\lambda_3)\vdash n$, a permutation $\varrho\in S_n$, and for certain  $q\in \{1,\dots,n\}$
we define operations $\iota_{1q},\iota_2,\iota_3$
on the highest weight vector
$w(x_1,x_2,x_3)=w_{\lambda}^{\varrho}(x_1,x_2,x_3)\in K\langle X_3\rangle$ 
that produce highest weight vectors in the degree $n+2$, $n+3$ or $n+4$  
homogeneous components of $K\langle X_3\rangle$:

\begin{itemize}
\item
If $\lambda_1>\lambda_2$ and the integer $q$ is at the $r$th position in the first row of the tableau $T_{\lambda}(\varrho)$, and $r>\lambda_2$,
then $w(x_1,x_2,x_3)$ has the form
\[
w(x_1,x_2,x_3)=\sum\pm u'x_1u'',
\]
where the summation runs on some monomials $u'$ and $u''$ of degree $q-1$ and $n-q$, respectively, and we define
\[
\iota_{1q}(w(x_1,x_2,x_3))=\sum\pm u'x_1^3u'';
\]
that is, $\iota_{1q}(w)=(w\cdot x_1^2)^\pi$, where 
\[\pi^{-1}=\left(\begin{matrix}1&\dots&q&q+1&q+2&\dots&n&n+1&n+2\\
1&\dots&q&q+3&q+4&\dots &n+2&q+1&q+2\\
\end{matrix}\right).\] 
\item
Let $\tau=(2,2)$ and let
\[
w_{(2,2)}^{(2)}(x_1,x_2)=\sum_{\sigma_1,\sigma_2\in S_2}\text{sign}(\sigma_1)\text{sign}(\sigma_2)
x_{\sigma_1(1)}x_{\sigma_2(1)}x_{\sigma_1(2)}x_{\sigma_2(2)}
\]
\[
=x_1^2x_2^2-x_1x_2^2x_1-x_2x_1^2x_2+x_2^2x_1^2
\]
be the highest weight vector corresponding to the $\tau$-tableau
\young(12,34) 
(i.e. $w_{(2,2)}^{(2)}=w_{\tau}^{\pi}=([x_1,x_2]^2)^{\pi}$ where $\pi$ is the transposition $\left(\begin{matrix}1&2&3&4\\
1&3&2&4\\
\end{matrix}\right)$).
Then we define
\[
\iota_2(w(x_1,x_2,x_3))=w(x_1,x_2,x_3)w_{(2,2)}^{(2)}(x_1,x_2).
\]
\item
We define
\[
\iota_3(w(x_1,x_2,x_3))=w(x_1,x_2,x_3)s_3(x_1,x_2,x_3).
\]
\end{itemize}

\begin{lemma}\label{adding even number of columns}
{\rm (i)} Let $\nu=(\nu_1,\nu_2,\nu_3)\vdash n$ and let $w(x_1,x_2,x_3)=w_{\nu}^{\varrho}(x_1,x_2,x_3)$
be the highest weight vector corresponding to the permutation $\varrho\in S_n$. Then the polynomials
$\iota_{1q}(w(x_1,x_2,x_3)),\iota_2(w(x_1,x_2,x_3)),\iota_3(w(x_1,x_2,x_3))$ are highest weight vectors of the form $w_{\mu}^{\sigma}$, where $\mu$ is the
partition $(\nu_1+2,\nu_2,\nu_3),(\nu_1+2,\nu_2+2,\nu_3),(\nu_1+1,\nu_2+1,\nu_3+1)$, respectively.

{\rm (ii)} 
Let $w^{(i)}(x_1,x_2,x_3)=w_{\nu}^{\varrho_i}(x_1,x_2,x_3)$, $\varrho_i\in S_n$, $i=1,\ldots,m$,
and let $\{a_1,a_2,a_3\}$ be the basis of $so_3(K)$ defined in the proof of 
Proposition~\ref{isomorphic algebras of invariants}.
If the matrices $w^{(i)}(a_1,a_2,a_3)$ are linearly independent in $M_3(K)$, then the matrices of each set
\[
\{\iota_{1q_i}(w^{(i)}(a_1,a_2,a_3))\mid i=1,\ldots,m\},\quad \{\iota_2(w^{(i)}(a_1,a_2,a_3))\mid i=1,\ldots,m\},
\]
\[
\{\iota_3(w^{(i)}(a_1,a_2,a_3))\mid i=1,\ldots,m\}\]
are also linearly independent in $M_3(K)$.
\end{lemma}

\begin{proof} (i) 
Applying $\iota_{1q}$ we insert $x_1^2$ between the $q^{\mathrm{th}}$ and $(q+1)^{\mathrm{st}}$ positions of the monomials of $w(x_1,x_2,x_3)$. 
So $\iota_{1q}(w_{\nu}^{\rho})=w_{\mu}^{\psi}$, where $\mu=(\nu_1+2,\nu_2,\nu_3)$ and the tableau $T_{\mu}(\psi)$ is obtained from the tableau $T_{\nu}(\rho)$ by adding $2$ to each entry greater than $q$, and writing $q+1$, $q+2$ in the two new boxes at the end of the first row of the Young diagram of $\mu$. 

Hence $\iota_{1q}(w(x_1,x_2,x_3))$ is a highest weight vector corresponding to the partition $(\nu_1+2,\nu_2,\nu_3)$.
Similarly, $\iota_2$ multiplies $w(x_1,x_2,x_3)$ by a 
highest weight vector of weight $(2,2)$, thus 
$\iota_2(w(x_1,x_2,x_3))$ is  a highest weight vector of weight  
$(\nu_1+2,\nu_2+2,\nu_3)$.
Finally, $\iota_3$ multiplies $w(x_1,x_2,x_3)$ by the standard polynomial $s_3(x_1,x_2,x_3)$ which is a 
highest weight vector of weight $(1,1,1)$, 
hence $\iota_3(w(x_1,x_2,x_3))$ is a highest weight vector with weight 
$(\nu_1+1,\nu_2+1,\nu_3+1)$.
It is also clear that the resulting highest weight vectors are all of the form $w_{\mu}^{\sigma}$ for some partition $\mu$ and permutation $\sigma$.

(ii) 
Direct computations show that
\[
a_1^3=\left(\begin{matrix}
0&-1&0\\
1&0&0\\
0&0&0\\
\end{matrix}\right)=-a_1,
\]
\[
w_{(2,2)}^{(2)}(a_1,a_2)=\left(\begin{matrix}
2&0&0\\
0&-1&0\\
0&0&-1\\
\end{matrix}\right),
\quad
s_3(a_1,a_2,a_3)=\left(\begin{matrix}
2&0&0\\
0&2&0\\
0&0&2\\
\end{matrix}\right).
\]
If the matrices $\iota_{1q_i}(w^{(i)}(a_1,a_2,a_3))$, $i=1,\ldots,m$, are linearly dependent, then the
equality $\iota_{1q_i}(w^{(i)}(a_1,a_2,a_3))=-w^{(i)}(a_1,a_2,a_3)$ implies the linear dependence for
$w^{(i)}(a_1,a_2,a_3)$ which is a contradiction.  Similarly, since the matrices
$w_{(2,2)}^{(2)}(a_1,a_2)$ and $s_3(a_1,a_2,a_3)$ are invertible, the linear dependence of
$\iota_2(w^{(i)}(a_1,a_2,a_3))$ and of $\iota_3(w^{(i)}(a_1,a_2,a_3))$, $i=1,\ldots,m$,
gives the linear dependence of $w^{(i)}(a_1,a_2,a_3)$.
\end{proof}

\begin{lemma}\label{linearly independent evaluations}
For each of the following partitions $\nu$ the evaluations of the highest weight vectors $w^{(i)}(a_1,a_2,a_3)$, $i=1,\ldots,m_{\nu}$,
are linearly independent if $m_{\nu}>1$ and nonzero if $m_{\nu}=1$:

{\rm (i)} For $\nu=(4,2)$ and the $\nu$-tableaux \young(1356,24), \young(1345,26), and \ \young(1236,45)
\[
w^{(1)}=[x_1,x_2]^2x_1^2,\quad w^{(2)}=[x_1,x_2](x_1^3x_2-x_2x_1^3),
\]
\[
w^{(3)}=(x_1^3x_2^2-x_1x_2x_1x_2x_1-x_2x_1^3x_2+x_2^2x_1^3)x_1;
\]

{\rm (ii)} For $\nu=(2,2)$ and the $\nu$-tableaux \young(13,24)  and \young(12,34) 
\[
w^{(1)}=[x_1,x_2]^2,\quad w^{(2)}=x_1^2x_2^2-x_1x_2^2x_1-x_2x_1^2x_2+x_2^2x_1^2;
\]

{\rm (iii)} For $\nu=(3,1,1)$ and the $\nu$-tableaux \young(145,2,3) and \  \young(134,2,5)
\[
w^{(1)}=s_3(x_1,x_2,x_3)x_1^2,\quad w^{(2)}=\sum_{\sigma\in S_3}\text{\rm sign}(\sigma)x_{\sigma(1)}x_{\sigma(2)}x_1^2x_{\sigma(3)};
\]

{\rm (iv)} For $\nu=(0)$, $w^{(1)}=1$;

{\rm (v)} For $\nu=(3,1)$ and the $\nu$-tableaux \young(134,2) and \  \young(312,4)
\[
w^{(1)}=[x_1,x_2]x_1^2,\quad w^{(2)}=x_1^2[x_1,x_2];
\]

{\rm (vi)} For $\nu=(1,1)$ and the $\nu$-tableau \young(1,2) $w^{(1)}=[x_1,x_2]$;

{\rm (vii)} For $\nu=(2,1)$ and the $\nu$-tableaux \young(13,2) and \  \young(21,3)
\[
w^{(1)}=[x_1,x_2]x_1,\quad w^{(2)}=x_1[x_1,x_2];
\]

{\rm (viii)} For $\nu=(3,2)$ and the $\nu$-tableaux \young(135,24) and \  \young(241,35)
\[
w^{(1)}=[x_1,x_2]^2x_1,\quad w^{(2)}=x_1[x_1,x_2]^2;
\]

{\rm (ix)} For $\nu=(1)$ and the $\nu$-tableau \young(1) $w^{(1)}=x_1$.
\end{lemma}

\begin{proof}
Direct computations show that:

(i) For $\nu=(4,2)$
\[
w^{(1)}(a_1,a_2,a_3)=\left(\begin{matrix}0&0&0\\
0&1&0\\
0&0&0\\
\end{matrix}\right),\quad
w^{(2)}(a_1,a_2,a_3)=\left(\begin{matrix}0&0&0\\
0&1&0\\
0&0&1\\
\end{matrix}\right),
\]
\[
w^{(3)}(a_1,a_2,a_3)=\left(\begin{matrix}-1&0&0\\
0&-1&0\\
0&0&0\\
\end{matrix}\right);
\]

{\rm (ii)} For $\nu=(2,2)$
\[
w^{(1)}(a_1,a_2,a_3)=\left(\begin{matrix}0&0&0\\
0&-1&0\\
0&0&-1\\
\end{matrix}\right),\quad
w^{(2)}(a_1,a_2,a_3)=\left(\begin{matrix}2&0&0\\
0&-1&0\\
0&0&-1\\
\end{matrix}\right);
\]

{\rm (iii)} For $\nu=(3,1,1)$
\[
w^{(1)}(a_1,a_2,a_3)=\left(\begin{matrix}-2&0&0\\
0&-2&0\\
0&0&0\\
\end{matrix}\right),\quad w^{(2)}(a_1,a_2,a_3)=\left(\begin{matrix}-1&0&0\\
0&-1&0\\
0&0&-2\\
\end{matrix}\right);
\]

{\rm (iv)} For $\nu=(0)$
\[
w^{(1)}(a_1,a_2,a_3)=\left(\begin{matrix}1&0&0\\
0&1&0\\
0&0&1\\
\end{matrix}\right);
\]

{\rm (v)} For $\nu=(3,1)$
\[
w^{(1)}(a_1,a_2,a_3)=\left(\begin{matrix}0&0&0\\
0&0&0\\
0&-1&0\\
\end{matrix}\right),
\quad w^{(2)}(a_1,a_2,a_3)=\left(\begin{matrix}0&0&0\\
0&0&1\\
0&0&0\\
\end{matrix}\right);
\]

{\rm (vi)} For $\nu=(1,1)$
\[
w^{(1)}(a_1,a_2,a_3)=\left(\begin{matrix}0&0&0\\
0&0&-1\\
0&1&0\\
\end{matrix}\right);
\]

{\rm (vii)} For $\nu=(2,1)$
\[
w^{(1)}(a_1,a_2,a_3)=\left(\begin{matrix}0&0&0\\
0&0&0\\
-1&0&0\\
\end{matrix}\right),\quad w^{(2)}(a_1,a_2,a_3)=\left(\begin{matrix}0&0&-1\\
0&0&0\\
0&0&0\\
\end{matrix}\right);
\]

{\rm (viii)} For $\nu=(3,2)$
\[
w^{(1)}(a_1,a_2,a_3)=\left(\begin{matrix}0&0&0\\
1&0&0\\
0&0&0\\
\end{matrix}\right),\quad w^{(2)}(a_1,a_2,a_3)=\left(\begin{matrix}0&-1&0\\
0&0&0\\
0&0&0\\
\end{matrix}\right);
\]

{\rm (ix)} For $\nu=(1)$ and the $\nu$-tableau \young(1)
\[
w^{(1)}(a_1,a_2,a_3)=\left(\begin{matrix}0&1&0\\
-1&0&0\\
0&0&0\\
\end{matrix}\right).
\]
In all nine cases the matrices $w^{(i)}(a_1,a_2,a_3)$ are linearly independent if $m_{\nu}>1$
and nonzero if $m_{\nu}=1$.
\end{proof}

Now we state the main result of this section.

\begin{theorem}\label{main theorem}
Let $K$ be a field of characteristic $0$ and let $I(M_3(K),so_3(K))$ be the ideal of the weak polynomial identities
for the pair $(M_3(K),so_3(K))$. Then the cocharacter sequence of $I(M_3(K),so_3(K))$ is
\[
\chi_n(M_3(K),so_3(K))=\sum_{\nu\vdash n}m_{\nu}(M_3(K),so_3(K))\chi_{\nu},\quad\nu=(\nu_1,\nu_2,\nu_3),
\]
where the multiplicity $m_{\nu}(M_3(K),so_3(K))$  equals to $m_{\nu}$ from Lemma~\ref{semistandard tableaux from 5 to 3} for $\nu\notin\{ (2k,0,0)\mid k=1,2,\dots\}$, whereas  $m_{(2k,0,0)}(M_3(K),so_3(K))=1$.  
\end{theorem}

\begin{proof}
The multiplicities of the cocharacter sequence $\chi_n(M_3(K),so_3(K))$
are determined by the structure of the relatively free algebra $F_3(M_3(K),so_3(K))$
as a $\text{GL}_3(K)$-module and we shall work with the representations of $\text{GL}_3(K)$.

The case $\nu=(n)$ is trivial because the multiplicity of $W_3(n)$ in the free associative algebra $K\langle X_3\rangle$
is equal to 1 and the generator $x_1^n$ of $W_3(n)$ does not vanish in $(M_3(K),so_3(K))$.
Hence we may assume that $\nu_2>0$.
By Corollary \ref{bounds for cocharacters}
the multiplicities $m_{\nu}=m_{\nu}(M_3(K),so_3(K))$ are bounded from above by the multiplicities stated in the theorem.
Hence it is sufficient to show that for a given $\nu$ in $F_3(M_3(K),so_3(K))$ there exist at least
$m_{\nu}$ linearly independent highest weight vectors $w^{(i)}(x_1,x_2,x_3)$, $i=1,\ldots,m_{\nu}$.

Let $\nu_1\equiv\nu_2\equiv\nu_3\text{ (mod 2)}$ and $\nu_1>\nu_2>\nu_3$. Hence
\[
\nu=(\nu_3+2r_2+2r_1,\nu_3+2r_2,\nu_3),\quad r_1,r_2>0.
\]
By Lemma \ref{linearly independent evaluations} for the partition $(4,2)$
there exist three highest weight vectors $w^{(i)}(x_1,x_2,x_3)\in F_3(M_3(K),so_3(K))$, $i=1,2,3$,
with linearly independent evaluations $w^{(i)}(a_1,a_2,a_3)$, $i=1,2,3$.
Applying to them $r_1-1$ times suitable operations $\iota_{1q_i}$, $r_2-1$ times the operation $\iota_2$,
and $\nu_3$ times the operation $\iota_3$, by Lemma~\ref{adding even number of columns} we obtain three highest weight vectors for the partition $\nu$
which are linearly independent in $F_3(M_3(K),so_3(K))$.
Similarly, if $\nu_1=\nu_2>\nu_3$, i.e., $r_1=0$, $r_2>0$, we have two highest weight vectors corresponding to
the partition $(2,2)$ and apply to them $r_2-1$ times the operation $\iota_2$
and $\nu_3$ times the operation $\iota_3$. All other cases are handled in a similar way.
\end{proof}

Theorem~\ref{main theorem} gives also the cocharacter sequence of 
$I(M_3(K),\text{\rm ad}(sl_2(K)))$, thanks to the following: 

\begin{proposition} The cocharacter sequence of the ideal of weak polynomial identities of the pair $(M_3(K),\text{\rm ad}(sl_2(K)))$
is the same as the cocharacter sequence of $I(M_3(K),so_3(K))$. 
\end{proposition} 

\begin{proof} 
Over the algebraic closure $\bar K$ of $K$, the pair   
$(M_3(\bar K),so_3(\bar K))$ is isomorphic to the pair 
$(M_3(\bar K),\text{ad}(sl_2(\bar K)))$, and as is well known, 
the cocharacter sequence does not change on extending the characteristic zero base field. 
\end{proof} 

\begin{remark} The cocharacter sequence of the pair $(M_2(K),sl_2(K))$ was found by Procesi \cite{Procesi:2},
see also \cite[Exercise 12.6.12]{Drensky:2}: 
\[
\chi_n(M_2(K),sl_2(K))=\sum_{\lambda\vdash n}\chi_{\lambda}, \quad \lambda=(\lambda_1,\lambda_2,\lambda_3), \ n=0,1,2,\ldots.
\]
This is multiplicity free, unlike the cocharacter sequence of 
$I(M_3(\bar K),\text{ad}(sl_2(\bar K)))$ given in Theorem~\ref{main theorem}.
\end{remark}

\begin{problem}
Let $\varrho:sl_2(K)\to\text{End}_K(V_q)\cong M_q(K)$
be a $q$-dimensional irreducible representation of $sl_2(K)$, $q>2$
(or $q=\infty$).
Find the cocharacter sequence of the pair $(M_q(K),\varrho(sl_2(K)))$.
\end{problem}

\section{The difference between ${\mathcal{E}}_p$ and ${\mathcal{F}}_p$}\label{sec:embedding}

Denote by $\kappa:K\langle X_p\rangle \to {\mathcal{F}}_p$ the $K$-algebra surjection with
$x_k\mapsto t_k$, $k=1,\dots,p$.
Clearly $\ker(\kappa)=I(M_3(K),so_3(K))\cap K\langle X_p\rangle$.
We can therefore reformulate Theorem~\ref{main theorem} as follows:

\begin{theorem}\label{thm:main theorem 2}
For $p\ge 3$ we have the ${\mathrm{GL}}_p(K)$-module isomorphism
\[
{\mathcal{F}}_p\cong \bigoplus_{n=0}^{\infty}\bigoplus_{\nu=(\nu_1,\nu_2,\nu_3)\vdash n}
m_{\nu}(M_3(K),so_3(K))W_p(\nu),
\]
where the multiplicities $m_{\nu}(M_3(K),so_3(K))$ are given in Theorem~\ref{main theorem}.
For $p<3$ the summands labeled by partitions $\nu$ with more than $p$ non-zero parts have to be removed from the formula.
\end{theorem}

\begin{theorem}\label{thm:covariant-GL}
%\begin{itemize}
%\item[(i)]
{\rm (i)} The ${\mathrm{GL}}_p(K)$-module ${\mathcal{E}}_p$ decomposes as
\[
{\mathcal{E}}_p\cong \sum_{\nu=(\nu_1,\nu_2,\nu_3)}m_{\nu}W_p(\nu),
\]
where the value of $m_{\nu}$ is the same as in Lemma~\ref{semistandard tableaux from 5 to 3} for $p\ge 3$;
when $p<3$, the summands labeled by partitions with more than $p$ non-zero parts are removed.

%\item[(ii)]
{\rm (ii)} For all $p$ we have
\[{\mathcal{E}}_p={\mathcal{F}}_p\oplus\bigoplus_{k=1}^{\infty}\langle \mathrm{tr}(t_1^{2k})I\rangle_{{\mathrm{GL}}_p(K)}\] 
(where $I$ is the $3\times 3$ identity matrix).
%\end{itemize}
\end{theorem}

\begin{proof} 
Statement (i) follows from  Corollary~\ref{cor:weyl}, Proposition~\ref{prop:embedding} (ii), 
Lemma~\ref{lemma:a4a5} and Lemma~\ref{semistandard tableaux from 5 to 3}. 
Combining statement (i) with Theorem~\ref{thm:main theorem 2}  we get that the factor space 
${\mathcal{E}}_p/{\mathcal{F}}_p$ decomposes as 
\[{\mathcal{E}}_p/{\mathcal{F}}_p\cong \bigoplus_{k=1}^{\infty}W_p((2k)).\] 
The only highest weight vector in ${\mathcal{F}}_p$ with weight $(2k)$ is $t_1^{2k}$. 
In ${\mathcal{E}}_p$ we have also the highest weight vector $\mathrm{tr}(t_1^{2k})I$ with weight $(2k)$, 
and these two highest weight vectors are linearly independent over $K$, as one can easily see  by making the substitution
$t_1\mapsto \left(\begin{array}{ccc}0 & 1 & 0 \\-1 & 0 & 0 \\0 & 0 & 0\end{array}\right)$.
\end{proof}

\begin{remark} (i) The Cayley-Hamilton theorem gives that
\[
t_1^3-\frac{1}{2}\text{tr}(t_1^2)t_1=0\text{ and hence }t_1^4=\frac{1}{2}\text{tr}(t_1^2)t_1^2.
\]
This easily implies that
\[
\text{tr}(t_1^{2k})=\frac{1}{2^{k-1}}\text{tr}^k(t_1^2)=2(-1)^k((t^{(1)}_{12})^2+(t^{(1)}_{13})^2+(t^{(1)}_{23})^2)^k,\quad k\geq 1.
\]

(ii) The ${\mathrm{GL}}_p(K)$-module structure of the algebra of ${\mathrm{GL}}_2(K)$-equivariant polynomial maps $sl_2(K)^{\oplus p}\to M_2(K)$,
where ${\mathrm{GL}}_2(K)$ acts by (simultaneous) conjugation, is determined in \cite[Theorem 2.2]{Procesi:2}.
This turns out to be multiplicity free, unlike our ${\mathcal{E}}_p$.

(iii) It follows from Theorem~\ref{thm:covariant-GL} (ii) that $\mathrm{tr}(t_1t_2)t_3$
can be expressed as a $K$-linear combination of monomials in $t_1,t_2,t_3$. Indeed,
the explicit identity of this form is
\[
\mathrm{tr}(t_1t_2)\cdot t_3=t_1t_2t_3-t_2t_3t_1+t_3t_1t_2+t_2t_1t_3+t_3t_2t_1-t_1t_3t_2.
\]
\end{remark}

We close this section with  a description of the center  $C({\mathcal{E}}_p)$ and $C({\mathcal{F}}_p)$ 
of the algebra ${\mathcal{E}}_p$ and ${\mathcal{F}}_p$. 

\begin{proposition} \label{prop:central} 
For $p\ge 2$, the algebra $C({\mathcal{E}}_p)$ is isomorphic to the generic trace algebra $K[T_p]^{\text{SO}_3(K)}$. 
\end{proposition}  

\begin{proof}  
Denote by $\mathbb{F}$ the field of fractions of $K[T_p]$. 
Let $a_1,a_3\in so_3(K)$ be the matrices introduced in the proof of 
Proposition~\ref{isomorphic algebras of invariants}. Then 
$I$, $a_1$, $a_3$, $a_1^2$, $a_3^2$, $a_1a_3$, $a_3a_1$, $a_1^2a_3$, $a_3^2a_1$ are linearly independent over $K$. It follows that $I$, $t_1$, $t_2$, $t_1^2$, $t_2^2$, $t_1t_2$, $t_2t_1$, $t_1^2t_2$, $t_2^2t_1$ are linearly independent over $\mathbb{F}$ in $M_3(\mathbb{F})$; indeed, otherwise we could arrange the entries of the above $9$ matrices into a  $9\times 9$ matrix whose determinant would be the zero element of $K[T_p]$,  contrary to the fact that the substitution $t_1\mapsto a_1$, $t_2\mapsto a_3$ 
in this polynomial gives a non-zero value. 
So the above $9$ monomials in $t_1,t_2$  constitute an $\mathbb{F}$-vector space basis of $M_3(\mathbb{F})$. 
Take any element $c\in   C({\mathcal{E}}_p)$. The centralizer of $c$ in $M_3(\mathbb{F})$ contains 
$t_1$ and $t_2$, hence it contains the above $\mathbb{F}$-vector space basis of $M_3(\mathbb{F})$. Consequently, $c$ is central in $M_3(\mathbb{F})$, and thus $c$ is a scalar matrix. 
So $c=f\cdot I$ for some $f\in K[T_p]$. Taking into account that $c$ gives an $\mathrm{SO}_3(K)$-equivariant map from $so_3(K)^{\oplus p}\to M_3(K)$, we get that $f$ is an element 
of the generic trace algebra. 
\end{proof}

\begin{corollary} \label{cor:C(cov_p)}
(i) As a ${\mathrm{GL}}_p(K)$-module ($p\ge 2$), $C({\mathcal{E}}_p)$   decomposes as
\begin{equation}\label{center of E}
C({\mathcal{E}}_p)\cong \bigoplus W_p(\lambda),
\end{equation}
where the summation runs on all $\lambda=(\lambda_1,\lambda_2,\lambda_3)$ such that
$\lambda_1-\lambda_2\equiv\lambda_2-\lambda_3\equiv 0\text{ (mod }2)$ for $p\ge 3$; 
for $p<3$ the terms corresponding to partitions with more than $p$ non-zero parts should be omitted. 

(ii) For $p\ge 2$ the center $C({\mathcal{F}}_p)$ of ${\mathcal{F}}_p$ is the direct sum of all
$W_p(\lambda)$ from (\ref{center of E}) such that $\lambda_2>0$.
\end{corollary} 

\begin{proof} Statement (i) follows from Proposition~\ref{prop:central} and Corollary~\ref{decomposition of special orthogonal invariants}. 
Statement (ii) follows from (i) and Theorem~\ref{thm:covariant-GL} (ii). 
\end{proof}

By analogy with the notion of a weak polynomial identity one may define \emph{weak central polynomials for the pair $(R,L)$} 
as elements of the free algebra $K\langle X\rangle$ which take central values in $R$ when evaluated on $L$. 
Denote by $\chi_n^{\text{c}}(R,L)$ the $S_n$-character of  the factor space 
$P_n^{\text{c}}/(P_n^{\text{c}}\cap I(R,L))$, where $P_n^c$ is  
the space of multilinear weak central polynomials for the pair $(R,L)$, and 
call $\chi_n^{\text{c}}(R,L)$ \emph{the central cocharacter sequence  for the pair $(R,L)$}. 
We can restate the structure of $C({\mathcal{F}}_p)$ as a ${\mathrm{GL}}_p(K)$-module in the language of central cocharacter sequence as follows: 
\begin{theorem} 
\[
\chi_n^{\text{c}}(M_3(K),so_3(K))=\sum_{\lambda\vdash n}\chi_{\lambda},
\]
where $\lambda=(\lambda_1,\lambda_2,\lambda_3)$,   
$\lambda_2>0$, and 
$\lambda_1-\lambda_2\equiv\lambda_2-\lambda_3\equiv 0\text{ (mod }2)$. 
Moreover, writing $\lambda$ in the form
\[
\lambda=(2(\mu_1+\mu_2)+\lambda_3,2\mu_2+\lambda_3,\lambda_3),\text{ where }\mu_2+\lambda_3>0,
\]
the corresponding highest weight vector is 
\[
s_3^{\lambda_3}(x_1,x_2,x_3)([x_1,x_2]^2-
 \frac 13(x_1^2x_2^2-x_1x_2^2x_1-x_2x_1^2x_2+x_2^2x_1^2))^{\mu_2-1}w'_{\mu_1}(x_1,x_2),
\]
\[
w'_{\mu_1}(x_1,x_2)=[x_1,x_2]^2x_1^{2\mu_1}-[x_1,x_2](x_1^{2\mu_1+1}x_2-x_2x_1^{2\mu_1+1})
\]
\[
+(x_1^3x_2^2-x_1x_2x_1x_2x_1-x_2x_1^3x_2+x_2^2x_1^3)x_1^{2\mu_1-1},\text{ if }\mu_1,\mu_2>0;
\]
\[
s_3^{\lambda_3-1}(x_1,x_2,x_3)w''_{\mu_1}(x_1,x_2),
\]
\[
w''_{\mu_1}(x_1,x_2)=s_3(x_1,x_2,x_3)x_1^{2\mu_1}+2\sum_{\sigma\in S_3}\text{\rm sign}(\sigma)x_{\sigma(1)}x_{\sigma(2)}x_1^{2\mu_1}x_{\sigma(3)},
\text{ if }\mu_1>0, \mu_2=0; 
\]
\[
s_3^{\lambda_3}(x_1,x_2,x_3)([x_1,x_2]^2-\frac 13(x_1^2x_2^2-x_1x_2^2x_1-x_2x_1^2x_2+x_2^2x_1^2))^{\mu_2},\text{ if }\mu_1=0.
\]
\end{theorem}
\begin{proof} 
The statement on the central cocharacter sequence is a reformulation of  Corollary~\ref{cor:C(cov_p)} (ii). For each summand $\chi_{\lambda}$ of the central cocharacter the statement gives a multihomogeneous element of $K\langle X_3\rangle$ of 
$\mathbb{N}_0^3$-degree $\lambda$, moreover, this element is easily seen to be a highest weight vector (by the explanations in Section~\ref{subsec:rep-GL}). 
It remains to show that they are weak central polynomials for the pair $(M_3(K),so_3(K))$. 
This holds for $s_3(x_1,x_2,x_3)$ by  Theorem~\ref{theorem of Razmyslov}. 
One can check by direct computation (for example by substituting $x_i\mapsto t_i$) 
that 
$[x_1,x_2]^2-
 \frac 13(x_1^2x_2^2-x_1x_2^2x_1-x_2x_1^2x_2+x_2^2x_1^2)$ is a weak central polynomial for $(M_3(K),so_3(K))$, and that $w'_1,w''_1$ are weak central polynomials for $(M_3(K),so_3(K))$. Given that, we claim that for any $\mu_1>0$, 
 $w'_{\mu_1}(b_1,b_2)$ is a scalar matrix for any $b_1,b_2\in so_3(K)$. 
 It is sufficient to show this in the special case when $K=\mathbb{R}$, the field of real numbers. Moreover, the adjoint $\mathrm{SO}_3(\mathbb{R})$-orbit of $b_1$ contains a scalar multiple of the matrix $a_1$ (introduced in the proof of Proposition~\ref{isomorphic algebras of invariants}.  For any $g\in \mathrm{SO}_3(\mathbb{R})$ we have 
 $w'_{\mu_1}(gb_1g^{-1},gb_2g^{-1})=gw'_{\mu_1}(b_1,b_2)g^{-1}$. 
 Thus (taking into account the homogeneity of $w'_{\mu_1}$ in $x_1$) we get that 
 it is sufficient to show that $w'_{\mu_1}(a_1,b_2)$ is a scalar matrix. 
 Inspection of the explicit form of $w'_{\mu_1}(x_1,x_2)$ shows that 
 the equality $a_1^3=-a_1$ implies that for $\mu_1>0$ we have $w'_{\mu_1+1}(a_1,b_2)=-w'_{\mu_1}(a_1,b_2)$. Since $w'_1(a_1,b_2)$ is a scalar matrix, 
 we conclude that $w'_{\mu_1}(a_1,b_2)$ is a scalar matrix for all $\mu_1>0$. 
 Similar argument works for $w''_{\mu_1}$. 
\end{proof} 

\section{The module of covariants}\label{sec:covariants}

We saw above (cf. Corollary~\ref{cor:weyl}) that to a large extent, the analysis of  ${\mathcal{E}}_p$ for arbitrary $p$
can be reduced to the special case $p=3$.
Our aim is to describe ${\mathcal{E}}_3$ as a module
over the ring $K[T_3]^{\mathrm{SO}_3(K)}$.

We set
\[
C:=\iota({\mathcal{E}}_3)=(K[T_3,Z]^{\mathrm{SO}_3(K)})^{(\mathbb{N}_0,\mathbb{N}_0,\mathbb{N}_0,1)},\quad \text{ where }\iota \text{ is defined in Proposition~\ref{prop:embedding}}.
\]
As a special case of the Clebsch-Gordan rules, the space of $3\times 3$ matrices has the decomposition
\[
M_3(K)=KI\oplus so_3(K)\oplus M_3(K)^+_0
\]
as a direct sum of irreducible $\mathrm{SO}_3(K)$-invariant subspaces, where $I$ is the identity matrix
and $M_3(K)^+_0$ is the space of trace zero symmetric $3\times 3$ matrices. Accordingly we have the decomposition
\begin{equation}\label{eq:C=}
C=C_1\oplus C_2\oplus C_3
\end{equation}
into a direct sum of $K[so_3(K)^{\oplus 3}]^{\mathrm{SO}_3(K)}$-submodules (that are also ${\mathrm{GL}}_3(K)$-submodules) of $C$. Namely
\[
C_1= K[T_3]^{\mathrm{SO}_3(K)}\cdot \mathrm{tr}(z),
\]
\[
C_2= (K[T_3,Z^-]^{\mathrm{SO}_3(K)})^{(\mathbb{N}_0,\mathbb{N}_0,\mathbb{N}_0,1)},
\]
and
\[
C_3\cong  (K[T_3,Z^+_0]^{\mathrm{SO}_3(K)})^{(\mathbb{N}_0,\mathbb{N}_0,\mathbb{N}_0,1)},
\]
where
\begin{align}\label{eq:s,u}  Z^-=\{u_{ij}:=\frac 12 (z_{ij}-z_{ji}) & \mid 1\le i<j\le 3\},
\\    \notag
 Z^+_0=\{s_{ij}:=\frac 12 (z_{ij}+z_{ji}),\ & s_{kk}:=z_{kk}-\frac{1}{3}(z_{11}+z_{22}+z_{33})\mid
\\ \notag  & 1\le i\le j\le3,\ k=1,2,3\},
\end{align}
and $(\mathbb{N}_0,\mathbb{N}_0,\mathbb{N}_0,1)$ in the exponents above indicates that we take the component
consisting of the polynomials that have total degree one in the variables belonging to $Z$.

Denote by $P$ the subalgebra of $K[T_3]^{\mathrm{SO}_3(K)}$ generated by
$\mathrm{tr}(t_1^2)$, $\mathrm{tr}(t_2^2)$, $\mathrm{tr}(t_3^2)$,
$\mathrm{tr}(t_1t_2)$, $\mathrm{tr}(t_1t_3)$, $\mathrm{tr}(t_2t_3)$.
Note that the six generators are algebraically independent over $K$.
The algebra $K[T_3]^{\mathrm{SO}_3(K)}$ is a free
$P$-module of rank two, generated by $1$ and $\mathrm{tr}(t_1t_2t_3)$
(see Corollary~\ref{cor:tr(xyz)} (ii)).
It follows that the $\mathbb{N}_0^3$-graded Hilbert series (or in other words, the formal
${\mathrm{GL}}_3(K)$-character of $C_1$) is
\begin{equation}\label{eq:Hilbert_C_1}
H(C_1;\tau_1,\tau_2,\tau_3)
=\frac{1+\tau_1\tau_2\tau_3}{\prod_{1\le i\le j\le 3}(1-\tau_i\tau_j)}.
\end{equation}

\begin{proposition}\label{prop:hilbert_series}
The $\mathbb{N}_0^3$-graded Hilbert series of $C_2$ and $C_3$ are the following:
\begin{eqnarray}
\label{eq:Hilbert_C_2}
H(C_2;\tau_1,\tau_2,\tau_3)=\frac{(S_{(1)}+S_{(1,1)})(\tau_1,\tau_2,\tau_3)}{\prod_{1\le i\le j\le 3}(1-\tau_i\tau_j)}\\
\label{eq:Hilbert_C_3}
H(C_3;\tau_1,\tau_2,\tau_3)=\frac{(S_{(2)}+S_{(2,1)}
-S_{(2,2,1)}-S_{(2,2,2)})(\tau_1,\tau_2,\tau_3)}
{\prod_{1\le i\le j\le 3}(1-\tau_i\tau_j)}
\end{eqnarray}
where
\[
S_{(1)}(\tau_1,\tau_2,\tau_3)=\tau_1+\tau_2+\tau_3,
\]
\[
S_{(2)}(\tau_1,\tau_2,\tau_3)=\sum_{1\le i\le j\le 3}\tau_i\tau_j,
\]
\[
S_{(1,1)}(\tau_1,\tau_2,\tau_3)=\sum_{1\le i<j\le 3}\tau_i\tau_j,
\]
\[
S_{(2,1)}(\tau_1,\tau_2,\tau_3)=\sum_{i\neq j}\tau_i^2\tau_j +2\tau_1\tau_2\tau_3,
\]
\[
S_{(2,2,1)}(\tau_1,\tau_2,\tau_3)=\tau_1\tau_2\tau_3 S_{(1,1)}(\tau_1,\tau_2,\tau_3),
\]
\[
S_{(2,2,2)}(\tau_1,\tau_2,\tau_3)=S_{(1,1,1)}(\tau_1,\tau_2,\tau_3)^2
=(\tau_1\tau_2\tau_3)^2
\]
are Schur polynomials (the formal characters of the ${\mathrm{GL}}_3(K)$-modules
$W_3(1)$, $W_3(2)$, $W_3(1,1)$, $W_3(2,1)$, $W_3(2,2,1)$, $W_3(2,2,2)$).
\end{proposition}

\begin{proof} It is well known that the Hilbert series in question are independent of the characteristic zero base field $K$.
Therefore we may assume $K=\mathbb{C}$, the field of complex numbers. View the $SO_3(\mathbb{C})$-module $\mathbb{C}[T_3,Z_0^+]$ as an
$SL_2(\mathbb{C})$-module via the natural surjection $SL_2(\mathbb{C})\to SO_3(\mathbb{C})$
with kernel consisting of the $2\times 2$ identity matrix and its negative.
The maximal compact subgroup $SU_2(\mathbb{C})$ (the special unitary group) of  $SL_2(\mathbb{C})$
has the same subspace of invariants in $\mathbb{C}[T_3,Z_0^+]$ as $SL_2(\mathbb{C})$. We compute the Hilbert series of
$C_3=(\mathbb{C}[T_3,Z_0^+]^{(\mathbb{N}_0,\mathbb{N}_0,\mathbb{N}_0,1)})^{SU_2(\mathbb{C})}$ using standard methods.
Namely, it can be expressed by the Molien-Weyl formula and the Weyl integration formula
as an integral over a maximal torus of $SU_2(\mathbb{C})$ as follows. Consider the maximal torus
$\mathbb{T}=\{\left(\begin{array}{cc}\rho & 0 \\0 & \rho^{-1}\end{array}\right)\mid |\rho|=1\}$ in
$SU_2(\mathbb{C})$. The character of the multihomogenous components of
$\mathbb{C}[T_3,Z_0^+]^{(\mathbb{N}_0,\mathbb{N}_0,\mathbb{N}_0,1)}$ as a $\mathbb{T}$-module
is given by the series
\[
\frac{\rho^4+\rho^2+1+\rho^{-2}+\rho^{-4}}{\prod_{j=1}^3(1-\rho^2\tau_j)(1-\tau_j)
(1-\rho^{-2}\tau_j)}.
\]
The roots of $SU_2(\mathbb{C})$ are $\rho^2$ and $\rho^{-2}$, and the order of the Weyl group is $2$. Therefore
the Molien-Weyl formula combined with the Weyl integration formula yields
\begin{align*}H(C_3;\tau_1,\tau_2,\tau_3)=\frac{1}{2}\int_{|\rho|=1}\frac{(\rho^4+\rho^2+1+\rho^{-2}+\rho^{-4})(1-\rho^2)(1-\rho^{-2})}
{\prod_{j=1}^3(1-\rho^2\tau_j)(1-\tau_j)
(1-\rho^{-2}\tau_j)}\frac{\mathrm{d}\rho}{2\pi \mathrm i \rho}
\\ =\frac 12 \cdot \frac{1}{2\pi\mathrm{i}}
\int_{|\rho|=1}\frac{-\rho^{12}+\rho^{10}+\rho^2-1}
{\rho\prod_{j=1}^3(1-\rho^2\tau_j)(1-\tau_j)
(\rho^2-\tau_j)}\mathrm{d}\rho.
\end{align*}
The above integral can be evaluated by residue calculus. Suppose that
$\tau_1,\tau_2,\tau_3$ are non-zero complex numbers of absolute value less than $1$.
Then the integrand has poles inside the unit circle at $\rho=\pm\sqrt{\tau_k}$, $k=1,2,3$,
and at $\rho=0$.
The residue at $\pm\sqrt{\tau_k}$ is
\[
\frac{-\tau_k^6+\tau_k^5+\tau_k-1}{2\tau_k(1-\tau_k^2)(1-\tau_k)\prod_{j\in  \{1,2,3\}\setminus \{k\}}(1-\tau_k\tau_j)(1-\tau_j)(\tau_k-\tau_j)},
\]
whereas the residue of the integrand at $\rho=0$ is
\[
\frac{1}{\prod_{j=1}^3(1-\tau_j)\tau_j}.
\]
It follows that
\begin{align*}
&H(C_3;\tau_1,\tau_2,\tau_3)=\frac 12\left(\frac{1}{\prod_{j=1}^3(1-\tau_j)\tau_j}+\right.\\
 & \left. 2\sum_{k=1}^3\frac{-\tau_k^6+\tau_k^5+\tau_k-1}{2\tau_k(1-\tau_k^2)(1-\tau_k)
 \prod_{j\in  \{1,2,3\}\setminus \{k\}}(1-\tau_k\tau_j)(1-\tau_j)(\tau_k-\tau_j)} \right).
\end{align*}
Bringing to common denominator the summands on the right hand side and after some cancellations we obtain \eqref{eq:Hilbert_C_3}.

Similarly,
\begin{align*}
H(C_2;\tau_1,\tau_2,\tau_3)=\frac 12 \int_{|\rho|=1}\frac{(\rho^2+1+\rho^{-2})(1-\rho^2)(1-\rho^{-2})}{\prod_{j=1}^3(1-\rho^2\tau_j)(1-\tau_j)
(1-\rho^{-2}\tau_j)}\frac{\mathrm{d}\rho}{2\pi \mathrm i \rho}
\\=\frac 12 \cdot \frac{1}{2\pi \mathrm{i}} \int_{|\rho|=1}
\frac{-\rho^9+\rho^7+\rho^3-\rho}
{\prod_{j=1}^3(1-\rho^2\tau_j)(1-\tau_j)(\rho^2-\tau_j)}\mathrm{d}\rho
\\=\frac 12\sum_{k=1}^3\frac{-\tau_k^4+\tau_k^3+\tau_k-1}
{(1-\tau_k^2)(1-\tau_k)\prod_{j\in\{1,2,3\}\setminus \{k\}}(1-\tau_k\tau_j)(1-\tau_j)(\tau_k-\tau_j)},
\end{align*}
from which one gets \eqref{eq:Hilbert_C_2} after bringing the summands to common denominator and cancelling certain factors.
\end{proof}

\begin{remark}
It would be possible to derive Theorem~\ref{thm:covariant-GL} (i) (giving the multiplicities of the irreducible $\text{GL}_p(K)$-module
summands in ${\mathcal{E}}_p$) using Proposition~\ref{prop:hilbert_series} and Corollary~\ref{cor:weyl}.
\end{remark}

\begin{proposition}\label{prop:C_2C_3gens} Write $s$ for the symmetric trace zero matrix
whose entries above and in the diagonal are $s_{ij}$, $1\le i\le j\le 3$, and write $u$ for the
skew-symmetric matrix whose entries above the diagonal are $u_{ij}$, $1\le i<j\le 3$
(where $s_{ij}$, $u_{ij}$ were introduced in \eqref{eq:s,u}).
%\begin{itemize}
%item[(i)]
{\rm (i)} $\mathrm{tr}(t_1u)$ is a highest weight vector in the ${\mathrm{GL}}_3(K)$-module $C_2$
generating a ${\mathrm{GL}}_3(K)$-submodule isomorphic to $W_3(1)$.

%\item[(ii)]
{\rm (ii)} $\mathrm{tr}(t_1t_2u)$ is a highest weight vector in the ${\mathrm{GL}}_3(K)$-module $C_2$
generating a ${\mathrm{GL}}_3(K)$-submodule isomorphic to
$W_3(1,1)$.

%\item[(iii)]
{\rm (iii)} $\mathrm{tr}(t_1^2s)$ is a highest weight vector in the ${\mathrm{GL}}_3(K)$-module $C_3$
generating a ${\mathrm{GL}}_3(K)$-submodule isomorphic to
$W_3(2)$.

%\item[(iv)]
{\rm (iv)} $\mathrm{tr}([t_1^2,t_2]s)$ is a highest weight vector in the ${\mathrm{GL}}_3(K)$-module $C_3$
generating a ${\mathrm{GL}}_3(K)$-submodule isomorphic to
$W_3(2,1)$.
%\end{itemize}
\end{proposition}

\begin{proof}
The map $K\langle X_3\rangle \to K[T_3,Z_0^+]$,
$f(x_1,x_2,x_3)\mapsto \mathrm{tr}(f(t_1,t_2,t_3)s)$ is a ${\mathrm{GL}}_3(K)$-module
homomorphism. As explained in Section~\ref{subsec:rep-GL},
$x_1^2$ is a highest weight vector in  $K\langle X_3\rangle^{(2)}$ generating a ${\mathrm{GL}}_3(K)$-submodule isomorphic to $W_3(2)$, whereas
$[x_1^2,x_2]$  is a highest weight vector in  $K\langle X_3\rangle^{(3)}$ generating a ${\mathrm{GL}}_3(K)$-submodule isomorphic to $W_3(2,1)$.
Since the images of these highest weight vectors in $C_3$ are non-zero, they are also highest weight vectors as required.
Similarly, the map $K\langle X_3\rangle \to K[T_3,Z^-]$,
$f(x_1,x_2,x_3)\mapsto \mathrm{tr}(f(t_1,t_2,t_3)u)$ is a ${\mathrm{GL}}_3(K)$-module
homomorphism. Now $x_1\in K\langle X_3\rangle^{(1)}$ is a highest weight vector
generating a ${\mathrm{GL}}_3(K)$-module isomorphic to $W_3(1)$.
Also $\frac 12 (x_1x_2-x_2x_1)\in K\langle X_3\rangle^{(2)}$ is a highest weight vector
generating a ${\mathrm{GL}}_3(K)$-module isomorphic to $W_3(1,1)$, and its image in $C_2$ is $\mathrm{tr}(t_1t_2u)$,
since $t_1t_2+t_2t_1$ is a symmetric matrix, hence
$\mathrm{tr}((t_1t_2+t_2t_1)u)=0$.
\end{proof}

The ${\mathrm{GL}}_3(K)$-submodule  $\langle \mathrm{tr}(t_1u)\rangle_{{\mathrm{GL}}_3(K)}$ generated by $\mathrm{tr}(t_1u)$
has the $K$-vector space basis $\{\mathrm{tr}(t_1u),\mathrm{tr}(t_2u),\mathrm{tr}(t_3u)\}$, and
the ${\mathrm{GL}}_3(K)$-submodule  $\langle \mathrm{tr}(t_1t_2u)\rangle_{{\mathrm{GL}}_3(K)}$ generated by $\mathrm{tr}(t_1t_2u)$ has the
$K$-vector space basis $\{\mathrm{tr}(t_1t_2u),\mathrm{tr}(t_1t_3u),\mathrm{tr}(t_2t_3u)\}$.

\begin{proposition}\label{prop:C2}
$C_2$ is a rank $6$ free $P$-module generated by
$\mathrm{tr}(t_1u)$, $\mathrm{tr}(t_2u)$, $\mathrm{tr}(t_3u)$, $\mathrm{tr}(t_1t_2u)$, $\mathrm{tr}(t_1t_3u)$, $\mathrm{tr}(t_2t_3u)$.
\end{proposition}

\begin{proof}
The fact that the above $6$ elements generate $C_2$ as a $K[T_3]^{\mathrm{SO}_3(K)}$-module is an immediate consequence of
Corollary~\ref{cor:tr(xyz)}.
The following two relations hold by Theorem~\ref{second fundamental theorem} and the proof of
Proposition~\ref{isomorphic algebras of invariants}. They (together with relations obtained by permuting $t_1,t_2,t_3$)
show that the $6$ elements in the statement in fact generate $C_2$ as a $P$-module:
\begin{equation}\label{eq:a}
\mathrm{tr}(t_1t_2t_3)\mathrm{tr}(t_1u)=\mathrm{tr}(t_1^2)\mathrm{tr}(t_2t_3u)
-\mathrm{tr}(t_1t_2)\mathrm{tr}(t_1t_3u)+\mathrm{tr}(t_1t_3)\mathrm{tr}(t_1t_2u)\end{equation}
\begin{align}\label{eq:b} \mathrm{tr}(t_1t_2t_3)\mathrm{tr}(t_1t_2u)&=
\frac{1}{8}\left(\mathrm{tr}(t_1t_3)\mathrm{tr}(t_2^2)-\mathrm{tr}(t_1t_2)\mathrm{tr}(t_2t_3)\right)\mathrm{tr}(t_1u)\\ \notag
&+\frac 18 \left(\mathrm{tr}(t_1^2)\mathrm{tr}(t_2t_3)-\mathrm{tr}(t_1t_2)\mathrm{tr}(t_1t_3)\right)\mathrm{tr}(t_2u)
\\ \notag & +\frac 18\left(\mathrm{tr}(t_1t_2)^2-\mathrm{tr}(t_1^2)\mathrm{tr}(t_2^2)\right)\mathrm{tr}(t_3u)
\end{align}
Therefore denoting by $e_1,\dots,e_6$ the standard generators of the free $P$-module $P^{\oplus 6}$, we have a $P$-module surjection
\begin{align*}\mu:P^{\oplus 6}\to C_2, \
&e_1\mapsto \mathrm{tr}(t_1u), \ e_2\mapsto \mathrm{tr}(t_2u),\ e_3\mapsto \mathrm{tr}(t_3u),\
\\ &e_4\mapsto \mathrm{tr}(t_1t_2u),\ e_5\mapsto \mathrm{tr}(t_1t_3u),\ e_6\mapsto\mathrm{tr}(t_2t_3u).
\end{align*}
This is a homomorphism of graded $P$-modules, where we endow $P^{\oplus 6}$ with the grading given by
$\deg(e_1)=\deg(e_2)=\deg(e_3)=1$ and $\deg(e_4)=\deg(e_5)=\deg(e_6)=2$,
and $C_2$ is endowed with the standard grading coming from the action of the subgroup of scalar matrices in ${\mathrm{GL}}_3(K)$.
The Hilbert series of $P^{\oplus 6}$ is $\frac{3\tau+3\tau^2}{(1-\tau^2)^6}$, and by
Proposition~\ref{prop:hilbert_series} this agrees with the Hilbert series of $C_2$.
It follows that $\mu$ is an isomorphism.
\end{proof}

The ${\mathrm{GL}}_3(K)$-submodule  $\langle \mathrm{tr}(t_1^2s)\rangle_{{\mathrm{GL}}_3(K)}$ generated by $\mathrm{tr}(t_1^2s)$ has the basis
\begin{equation}\label{eq:e_ij}
e_{ij}:=\mathrm{tr}(t_it_js), \ 1\le i \le j\le 3.
\end{equation}
The ${\mathrm{GL}}_3(K)$-submodule  $\langle \mathrm{tr}([t_1^2,t_2]s)\rangle_{{\mathrm{GL}}_3(K)}$ generated by $\mathrm{tr}([t_1^2,t_2]s)$ has the basis
\begin{eqnarray}\label{eq:e_iij}&e_{iij}:=\mathrm{tr}([t_i^2,t_j]s), \ i\neq j\in\{1,2,3\},\\  \notag
& e_{132}:=\mathrm{tr}([t_1t_3+t_3t_1,t_2]s), \
e_{123}:=\mathrm{tr}([t_1t_2+t_2t_1,t_3]s).
\end{eqnarray}

\begin{theorem}\label{thm:C3}
%\begin{itemize}
%\item[(i)]
{\rm (i)} As a $P$-module, $C_3$ is generated by
\[
\{e_{iij},e_{132},e_{123},e_{kl}\mid i\neq j\in\{1,2,3\},\ 1\le k\le l\le 3\}.
\]
Moreover, it has the direct sum decomposition
\[
C_3=C_3^{(0)}\oplus C_3^{(1)},\ \text{ where }\
C_3^{(0)}=P\cdot \langle e_{11}\rangle_{{\mathrm{GL}}_3(K)},\
C_3^{(1)}=P\cdot \langle e_{112}\rangle_{{\mathrm{GL}}_3(K)}.
\]

%\item[(ii)]
{\rm (ii)} The $P$-module $C_3^{(0)}$ has the free resolution
\[
0\longrightarrow P\stackrel{\psi^{(0)}}\longrightarrow P^{\oplus 6}\stackrel{\varphi^{(0)}}\longrightarrow  C_3^{(0)}\longrightarrow 0
\]
where denoting by $e_1,e_2,e_3,e_4,e_5,e_6$ the standard generators of $P^6$,
$\varphi^{(0)}$ is the $P$-module homomorphism given by
\[
\varphi^{(0)}: e_1\mapsto e_{11}, e_2\mapsto e_{12}, e_3\mapsto e_{13}, e_4\mapsto e_{22}, e_5\mapsto e_{23}, e_6\mapsto e_{33},
\]
and $\psi^{(0)}$ maps the generator of the rank one $P$-module $P$ to
\begin{equation}\label{eq:ker phi^0}
\left(\begin{array}{c}\frac 12(\mathrm{tr}(t_2^2)\mathrm{tr}(t_3^2)-\mathrm{tr}(t_2t_3)^2) \\\mathrm{tr}(t_1t_3)\mathrm{tr}(t_2t_3)-\mathrm{tr}(t_1t_2)\mathrm{tr}(t_3^2) \\
\mathrm{tr}(t_1t_2)\mathrm{tr}(t_2t_3)-\mathrm{tr}(t_1t_3)\mathrm{tr}(t_2^2) \\\frac 12(\mathrm{tr}(t_1^2)\mathrm{tr}(t_3^2)-\mathrm{tr}(t_1t_3)^2) \\
\mathrm{tr}(t_1t_2)\mathrm{tr}(t_1t_3)-\mathrm{tr}(t_1^2)\mathrm{tr}(t_2t_3) \\\frac 12(\mathrm{tr}(t_1^2)\mathrm{tr}(t_2^2)-\mathrm{tr}(t_1t_2)^2)\end{array}\right)\in P^{\oplus 6}.
\end{equation}

%\item[(iii)]
{\rm (iii)} The $P$-module $C_3^{(1)}$ has the free resolution
\[
0\longrightarrow P^{\oplus 3}\stackrel{\psi^{(1)}}\longrightarrow
P^{\oplus 8}\stackrel{\varphi^{(1)}}\longrightarrow  C_3^{(1)}\longrightarrow 0
\]
where denoting by $e_1,e_2,e_3,e_4,e_5,e_6,e_7,e_8$ the standard generators of $P^8$,
$\varphi^{(1)}$ is the $P$-module homomorphism given by
\begin{eqnarray*}
\varphi^{(1)}:& e_1\mapsto e_{112}, e_2\mapsto e_{221}, e_3\mapsto e_{113}, e_4\mapsto e_{331},\\
& e_5\mapsto e_{223}, e_6\mapsto e_{332},e_7\mapsto e_{132},
e_8\mapsto e_{123}
\end{eqnarray*}
and $\psi^{(1)}:P^{\oplus 3}\to P^{\oplus 8}$ is given by the matrix below:
\begin{equation}\label{eq:ker phi^1}
\left(\begin{array}{ccc}\mathrm{tr}(t_2t_3) & 0 & -\mathrm{tr}(t_3^2) \\\mathrm{tr}(t_1t_3) & -\mathrm{tr}(t_3^2) & 0 \\-\mathrm{tr}(t_2^2) & 0 & \mathrm{tr}(t_2t_3) \\
0 & -\mathrm{tr}(t_2^2) & \mathrm{tr}(t_1t_2) \\-\mathrm{tr}(t_1^2) & \mathrm{tr}(t_1t_3) & 0 \\0 & \mathrm{tr}(t_1t_2) & -\mathrm{tr}(t_1^2) \\
0 & -\mathrm{tr}(t_2t_3) & \mathrm{tr}(t_1t_3) \\\mathrm{tr}(t_1t_2) & -\mathrm{tr}(t_2t_3) & 0\end{array}\right)
\in P^{8\times 3}
\end{equation}
%\end{itemize}
\end{theorem}

\begin{proof} (i)
$C_3$ is spanned as a $K$-vector space by products
\[
\mathrm{tr}(t_{i_1}\cdots t_{i_k})\cdots\mathrm{tr}(t_{j_1}\cdots t_{j_l})\mathrm{tr}(t_{a_1}\cdots t_{a_m}s)
\]
by Theorem~\ref{orthogonal trace invariants} and
Theorem~\ref{special orthogonal trace invariants}.
For $k\ge 4$, $\mathrm{tr}(t_{i_1}\cdots t_{i_k})$ can be expressed as a polynomial in
$\mathrm{tr}(t_it_j)$ and $\mathrm{tr}(t_it_jt_k)$ by 
Corollary~\ref{cor:tr(xyz)} (ii). 
Moreover, $\mathrm{tr}(t_it_jt_k)$ is non-zero only if
$i,j,k$ are distinct.

\medskip
\noindent{\bf Claim:}  for $k\ge 4$, $\mathrm{tr}(t_{i_1}\cdots t_{i_k}s)$ can be expressed by
products of traces of shorter products.
\medskip

Indeed, one can easily verify
the identity
\begin{equation}\label{eq:tr(t_1t_2t_3t_4s)}
\mathrm{tr}(t_1t_2t_3t_4s)
=\frac 12(\mathrm{tr}(t_1t_2)\mathrm{tr}(t_3t_4s)+\mathrm{tr}(t_3t_4)\mathrm{tr}(t_1t_2s)-\mathrm{tr}(t_1t_4)\mathrm{tr}(t_2t_3s)),
\end{equation}
implying our claim for $k=4$.
Apply next the fundamental trace identity (see for example \cite[p. 63, Theorem 5.2.4] {Drensky-Formanek}) for the four $3\times 3$ matrices
$t_1t_2$, $t_3t_4$, $t_5$, $s$, and take into acount that $0=\mathrm{tr}(t_i)=\mathrm{tr}(s)=\mathrm{tr}(t_is)$ to get
\begin{eqnarray}\label{eq:tr(t_1t_2t_3t_4t_5s)}
0&=&\mathrm{tr}(t_1t_2t_3t_4t_5s)+\mathrm{tr}(t_3t_4t_5t_1t_2s)+\mathrm{tr}(t_5t_1t_2t_3t_4s)\\ \notag
& & +\mathrm{tr}(t_3t_4t_1t_2t_5s)+\mathrm{tr}(t_1t_2t_5t_3t_4s)+\mathrm{tr}(t_5t_3t_4t_1t_2s)\\ \notag
& & -\mathrm{tr}(t_1t_2)\mathrm{tr}(t_3t_4t_5s)-\mathrm{tr}(t_1t_2)\mathrm{tr}(t_5t_3t_4s)-\mathrm{tr}(t_3t_4)\mathrm{tr}(t_1t_2t_5s)\\ \notag
& & -\mathrm{tr}(t_3t_4)\mathrm{tr}(t_5t_1t_2s)-\mathrm{tr}(t_1t_2t_5)\mathrm{tr}(t_3t_4s)-\mathrm{tr}(t_3t_4t_5)(\mathrm{tr}(t_1t_2s).
\end{eqnarray}
For $f,g\in C_3$ write  $f\equiv g$ if $f-g\in K[T_3]^{\mathrm{SO}_3(K)}_+ C_3$,
where $K[T_3]^{\mathrm{SO}_3(K)}_+$ stands for the sum of the positive degree homogeneous components of $K[T_3]^{\mathrm{SO}_3(K)}$.
Since $[t_i,t_j]$ is a skew-symmetric matrix, the identity \eqref{eq:tr(t_1t_2t_3t_4s)}
implies that
\[
\mathrm{tr}(t_1t_2t_3t_4t_5s)\equiv \mathrm{tr}(t_{\pi(1)}t_{\pi(2)}t_{\pi(3)}t_{\pi(4)}t_{\pi(5)}s)
\text{ for any permutation }\pi\in S_5.
\]
Therefore \eqref{eq:tr(t_1t_2t_3t_4t_5s)} implies $6\mathrm{tr}(t_1t_2t_3t_4t_5s)\equiv 0$.
This settles our claim for $k=5$.
Finally, for $k\ge 6$, recall that $\mathrm{tr}(z_1z_2z_3z_4z_5z_6z_7)$ can be expressed by
traces of shorter products where $z_1,\dots,z_7$ are arbitrary (not necessarily skew-symmetric or symmetric) $3\times 3$ matrices
(see for example \cite[p. 78, Theorem 6.1.6 and p. 79]{Drensky-Formanek}), so our Claim holds for $k\ge 6$ as well.

Thus we proved that $C_3$ is generated as a $K[T_3]^{\mathrm{SO}_3(K)}$-module by
\[
V:=\mathrm{Span}_{K}\{\mathrm{tr}(t_it_js),\ \mathrm{tr}(t_it_jt_ks)\mid i,j,k\in\{1,2,3\}\}.
\]
This is a ${\mathrm{GL}}_3(K)$-submodule of $C_3$.
Consider the surjective ${\mathrm{GL}}_3(K)$-module homomorphism
$\rho:K\langle X_3\rangle^{(2)}\oplus  K\langle X_3\rangle^{(3)}
\to V$ given by $\rho(f(x_1,x_2,x_3))=\mathrm{tr}(f(t_1,t_2,t_3)s)$. As a ${\mathrm{GL}}_3(K)$-module,
$K\langle X_3\rangle^{(2)}$ is generated by $x_1^2$ and $[x_1,x_2]$, whereas
$K\langle X_3\rangle^{(3)}$ is generated by $x_1^3$, $[x_1^2,x_2]$,
$[x_1,[x_1,x_2]]$,  $s_3(x_1,x_2,x_3)=
\sum_{\pi\in S_3}\mathrm{sign}(\pi)x_{\pi(1)}x_{\pi(2)}x_{\pi(3)}$.
Now $\rho([x_1,x_2])$, $\rho(x_1^3)$, $\rho([x_1,[x_1,x_2]])$,
$\rho(s_3(x_1,x_2,x_3))$ are all zero. Hence we conclude
\[
V=\langle e_{11}\rangle_{{\mathrm{GL}}_3(K)}\oplus
\langle e_{112} \rangle_{{\mathrm{GL}}_3(K)}.
\]
Recall that $K[T_3]^{\mathrm{SO}_3(K)}$ is a rank two free $P$-module generated by $1$ and $\mathrm{tr}(t_1t_2t_3)$ by Theorem~\ref{first fundamental theorem},
Theorem~\ref{second fundamental theorem} and Proposition~\ref{isomorphic algebras of invariants}.
Thus by $C_3=K[T_3]^{\mathrm{SO}_3(K)}\cdot V$ we conclude that
$C_3$ is generated as a $P$-module by $V+\mathrm{tr}(t_1t_2t_3)V$.
Next we show that
\begin{equation} \label{eq:tr(t_1t_2t_3)V}
\mathrm{tr}(t_1t_2t_3)V\subseteq PV.
\end{equation}
Indeed, observe that $\mathrm{tr}(t_1t_2t_3)$
spans a $1$-dimensional ${\mathrm{GL}}_3(K)$-invariant subspace in $K[T_3,Z^+_0]$. Therefore to prove \eqref{eq:tr(t_1t_2t_3)V}, it suffices to show that
the ${\mathrm{GL}}_3(K)$-module generators $e_{11}$ and $e_{112}$ are multiplied by
$\mathrm{tr}(t_1t_2t_3)$ into $PV$. This follows from the following two equalities:
\begin{equation}\label{eq:c}
\mathrm{tr}(t_1t_2t_3)e_{11}=
\frac 14 \mathrm{tr}(t_1t_3)e_{112}-\frac 14\mathrm{tr}(t_1t_2)e_{113}
-\frac 1{12} \mathrm{tr}(t_1^2)e_{132}+\frac 1{12}\mathrm{tr}(t_1^2)e_{123}
\end{equation}
\begin{align}\label{eq:d} \mathrm{tr}(t_1t_2t_3)e_{112}
&= \frac 12\left(\mathrm{tr}(t_1t_3)\mathrm{tr}(t_2^2)-\mathrm{tr}(t_1t_2)\mathrm{tr}(t_2t_3)\right)e_{11} \\ \notag
& +\frac 12\left(\mathrm{tr}(t_1^2)\mathrm{tr}(t_2t_3)-\mathrm{tr}(t_1t_2)\mathrm{tr}(t_1t_3)\right)e_{12} \\ \notag
& +\frac 12 \left(\mathrm{tr}(t_1t_2)^2-\mathrm{tr}(t_1^2)\mathrm{tr}(t_2^2)\right)e_{13}
\end{align}
So we proved
\[
C_3=P\langle e_{11}\rangle_{{\mathrm{GL}}_3(K)}+P\langle e_{112} \rangle_{{\mathrm{GL}}_3(K)}.
\]
The above sum is necessarily direct, as the polynomials in the first summand have odd total degree,
whereas the polynomials in the second summand have even total degree.
This finishes the proof of (i).

(ii) We proved above that $\varphi^{(0)}$ is surjective onto $C_3^{(0)}$.
Using \cite{cocoa} we found the following relation:
\begin{eqnarray*}
0&=&\frac 12(\mathrm{tr}(t_2^2)\mathrm{tr}(t_3^2)-\mathrm{tr}(t_2t_3)^2)e_{11}
+(\mathrm{tr}(t_1t_3)\mathrm{tr}(t_2t_3)-\mathrm{tr}(t_1t_2)\mathrm{tr}(t_3^2))e_{12} \\
& & (\mathrm{tr}(t_1t_2)\mathrm{tr}(t_2t_3)-\mathrm{tr}(t_1t_3)\mathrm{tr}(t_2^2))e_{13}
+\frac 12 (\mathrm{tr}(t_1^2)\mathrm{tr}(t_3^2)-\mathrm{tr}(t_1t_3)^2)e_{22}
\\ & & (\mathrm{tr}(t_1t_2)\mathrm{tr}(t_1t_3)-\mathrm{tr}(t_1^2)\mathrm{tr}(t_2t_3))e_{23}
+\frac 12 (\mathrm{tr}(t_1^2)\mathrm{tr}(t_2^2)-\mathrm{tr}(t_1t_2)^2)e_{33}
\end{eqnarray*}
Hence we have established $\psi^{(0)}(P)\subseteq \ker(\varphi^{(0)})$.
Taking into account the Hilbert series of $C_3^{(0)}$ we may conclude the equality
$\psi^{(0)}(P)= \ker(\varphi^{(0)})$. Indeed, by Proposition~\ref{prop:hilbert_series}
we have that the univariate Hilbert series of $C_3$ with the standard $\mathbb{N}_0$-grading
(coming from the action of the subgroup of scalar matrices in ${\mathrm{GL}}_3(K)$) is
\[
\frac{6\tau^2-\tau^6}{(1-\tau^2)^6}.
\]
The Hilbert series of the free module $P^{\oplus 6}$ (endowed with the appropriate grading respected by $\varphi^{(0)}$ is $\frac{6\tau^2}{(1-\tau^2)^6}$.
It follows that the Hilbert series of $\ker(\varphi^{(0)})$ is $\frac{\tau^6}{(1-\tau^2)^6}$,
which obviously agrees with the Hilbert series of the rank one free $P$-submodule
$\psi^{(0)}(P)$ generated by a single element of degree $6$.

(iii) In the proof of (i) we saw already that $\varphi^{(1)}$ is surjective onto $C_3^{(1)}$.
Using \cite{cocoa} we found the relation
\begin{equation}\label{eq:221}
0=\mathrm{tr}(t_2t_3)e_{112}+\mathrm{tr}(t_1t_3)e_{221}-\mathrm{tr}(t_1^2)e_{223}
-\mathrm{tr}(t_2^2)e_{113}+\mathrm{tr}(t_1t_2)e_{123}.
\end{equation}
This means that the first column of the $8\times 3$ matrix in the statement (iii) belongs to
$\ker(\varphi^{(1)})$. Permuting cyclically the matrix variables $t_1,t_2,t_3$ in \eqref{eq:221}
we get two other relations, meaning that the second and third columns
of the $8\times 3$ matrix in \eqref{eq:ker phi^1}
belong to
$\ker(\varphi^{(1)})$. So we have $\psi^{(1)}(P^{\oplus 3})\subseteq \ker(\varphi^{(1)})$.
 As the upper $3\times 3$ minor of the $8\times 3$ matrix in \eqref{eq:ker phi^1}
 has non-zero determinant, we get that $\psi^{(1)}$ is injective, and consequently
 the univariate Hilbert series of $\psi^{(1)}(P^{\oplus 3})$ agrees with
 $\frac{3\tau^5}{(1-\tau^2)^6}$, the Hilbert series of $P^{\oplus 3}$ (graded appropriately).
 On the other hand, by Proposition~\ref{prop:hilbert_series} we know that the Hilbert series of $C_3^{(1)}$
 is $\frac{8\tau^3-3\tau^5}{(1-\tau^2)^6}$. the Hilbert series of $P^{\oplus 8}$ (with the suitable grading)
 is $\frac{8\tau^3}{(1-\tau^2)^6}$, implying that the Hilbert series of
 $\ker(\varphi^{(1)})$ is $\frac{3\tau^5}{(1-\tau^2)^6}$, the same as the Hilbert series
 of $\psi^{(1)}(P^{\oplus 3})$.
This proves the equality $\mathrm{im}(\psi^{(1)})=\ker(\varphi^{(1)})$.
\end{proof}

\begin{theorem}\label{thm:cov_3}
%\begin{itemize}
%\item[(i)]
{\rm (i)} The $P$-module ${\mathcal{E}}_3$ has the direct sum decomposition
\[
{\mathcal{E}}_3={\mathcal{E}}_{3,1}\oplus {\mathcal{E}}_{3,2}^{(1)}\oplus{\mathcal{E}}_{3,2}^{(0)}\oplus {\mathcal{E}}_{3,3}^{(0)}\oplus{\mathcal{E}}_{3,3}^{(1)},
\]
where
\begin{align*}
& {\mathcal{E}}_{3,1}=P\cdot I\oplus P\cdot \mathrm{tr}(t_1t_2t_3)I=K[T_3]^{SO_3(K)}\cdot I\subset {\mathcal{E}}_3 \\
& {\mathcal{E}}_{3,2}^{(1)}=P\cdot \langle t_1\rangle_{GL_3(K)} \\
& {\mathcal{E}}_{3,2}^{(0)}= P\cdot \langle [t_1,t_2]\rangle_{GL_3(K)}
\\ & {\mathcal{E}}_{3,3}^{(0)}=P\cdot \langle t_1^2-\frac 13\mathrm{tr}(t_1^2)I\rangle_{GL_3(K)}
\\ & {\mathcal{E}}_{3,3}^{(1)}=P\cdot \langle [t_1^2,t_2]\rangle_{GL_3(K)}.
\end{align*}

%\item[(ii)]
{\rm (ii)} Both ${\mathcal{E}}_{3,2}^{(1)}$ and ${\mathcal{E}}_{3,2}^{(0)}$ are free $P$-modules of rank $3$:
\[
{\mathcal{E}}_{3,2}^{(1)}=P\cdot t_1\oplus P\cdot t_2\oplus P\cdot t_3 \text{ and }
{\mathcal{E}}_{3,2}^{(0)}=
P\cdot [t_1,t_2]\oplus P\cdot [t_1,t_3]\oplus P\cdot [t_2,t_3].
\]

%\item[(iii)]
{\rm (iii)} The $K$-vector space $\langle t_1^2-\frac 13\mathrm{tr}(t_1^2)I\rangle_{GL_3(K)}$ has the basis
\[
\{f_{ij}=\frac 12 (t_it_j+t_jt_i)-\frac 13 \mathrm{tr}(t_it_j)I\mid 1\le i\le j\le 3\},
\]
and the $P$-module ${\mathcal{E}}_{3,3}^{(0)}$ has the free resolution
\[
0\longrightarrow P\stackrel{\mu^{(0)}}\longrightarrow P^{\oplus 6}\stackrel{\eta^{(0)}}\longrightarrow  C_3^{(0)}\longrightarrow 0,
\]
where denoting by $e_1,e_2,e_3,e_4,e_5,e_6$ the standard generators of $P^6$,
$\eta^{(0)}$ is the $P$-module surjection given by
\[
\eta^{(0)}: e_1\mapsto f_{11}, e_2\mapsto f_{12}, e_3\mapsto f_{13}, e_4\mapsto f_{22}, e_5\mapsto f_{23}, e_6\mapsto f_{33},
\]
and $\mu^{(0)}$ maps the generator of the rank one $P$-module $P$ to the element of
$P^{\oplus 6}$ given in \eqref{eq:ker phi^0} in Theorem~\ref{thm:C3} (ii).

%\item[(iv)]
{\rm (iv)} The $K$-vector space $\langle [t_1^2,t_2]\rangle_{GL_3(K)}$ has the  basis
\[
\{f_{iij}=[t_i^2,t_j],\ f_{132}=[t_1t_3+t_3t_1,t_2],\ f_{123}=[t_1t_2+t_2t_1,t_3]\mid i\neq j\in\{1,2,3\}\},
\]
and the $P$-module $C_3^{(1)}$ has the free resolution
\[
0\longrightarrow P^{\oplus 3}\stackrel{\mu^{(1)}}\longrightarrow
P^{\oplus 8}\stackrel{\eta^{(1)}}\longrightarrow  C_3^{(1)}\longrightarrow 0
\]
where denoting by $e_1,e_2,e_3,e_4,e_5,e_6,e_7,e_8$ the standard generators of $P^8$,
$\eta^{(1)}$ is the $P$-module surjection given by
\begin{eqnarray*}
\eta^{(1)}:& e_1\mapsto f_{112}, e_2\mapsto f_{221}, e_3\mapsto f_{113}, e_4\mapsto f_{331},\\
& e_5\mapsto f_{223}, e_6\mapsto f_{332},e_7\mapsto f_{132},
e_8\mapsto f_{123}
\end{eqnarray*}
and $\mu^{(1)}:P^{\oplus 3}\to P^{\oplus 8}$ is given by the matrix in \eqref{eq:ker phi^1}
in Theorem~\ref{thm:C3} (iii).
%\end{itemize}
\end{theorem}

\begin{proof}
Consider the ${\mathrm{GL}}_3(K)$-module isomorphism
\[
\iota:{\mathcal{E}}_3\to (K[T_3,Z]^{SO_3(K)})^{(\mathbb{N}_0,\mathbb{N}_0,\mathbb{N}_0,1)},
\quad f\mapsto \mathrm{tr}(fz)
\]
from Proposition~\ref{prop:embedding} (ii).
Write the generic matrix $z$ as the sum
\[
z=\frac 13 \mathrm{tr}(z)I+s+u,\text{ with }s,u\text{ as in Proposition~\ref{prop:C_2C_3gens}}.
\]
We have the equalities
\begin{align*} & 0=\mathrm{tr}(t_i)=\mathrm{tr}([t_i,t_j])=\mathrm{tr}(t_i^2-\frac 13\mathrm{tr}(t_i^2)I)=\mathrm{tr}([t_i^2,t_j])
\\ & 0=\mathrm{tr}(s)=\mathrm{tr}(t_is)=\mathrm{tr}([t_i,t_j]s)
\\ & 0=\mathrm{tr}(u)=\mathrm{tr}(t_i^2u)=\mathrm{tr}([t_i^2,t_j]u)=\mathrm{tr}((t_it_j+t_jt_i)u).
\end{align*}
These equalities show that
\begin{align*}
& \iota(t_i)=\mathrm{tr}(t_iu) \ (i=1,2,3),
\\ & \iota([t_i,t_j])=\mathrm{tr}([t_i,t_j]u)=2\mathrm{tr}(t_it_ju) \ (1\le i<j\le 3\}
\\ & \iota(\frac 12(t_it_j+t_jt_i))=\mathrm{tr}(t_it_js)+\mathrm{tr}(t_it_j)\mathrm{tr}(z),\ (1\le i\le j\le 3).
\end{align*}
Moreover, we have
\begin{align*} \iota(f_{ij})=e_{ij}\ (1\le i\le j\le 3),\
\\ \iota(f_{iij})=e_{iij} \ (i\neq j\in\{1,2,3\}),
\\ \iota(f_{132})=e_{132},\ \iota(f_{123})=e_{123}
\end{align*}
(where $e_{ij}$, $e_{iij}$, $e_{132}$, $e_{123}$ were defined in \eqref{eq:e_ij}, \eqref{eq:e_iij}). 
Since $\iota$ is a $P$-module homomorphism, 
it follows that $\iota$ restricts to isomorphisms
${\mathcal{E}}_{3,1}\stackrel{\cong}\longrightarrow C_1$,
${\mathcal{E}}_{3,2}^{(1)}+{\mathcal{E}}_{3,2}^{(0)}\stackrel{\cong}\longrightarrow C_2$,
${\mathcal{E}}_{3,3}^{(0)}\stackrel{\cong}\longrightarrow C_3^{(0)}$,
and ${\mathcal{E}}_{3,3}^{(1)}\stackrel{\cong}\longrightarrow C_3^{(1)}$.
Thus our statements immediately follow from \eqref{eq:C=}, Corollary~\ref{cor:tr(xyz)} (ii),
Proposition~\ref{prop:C2}, and Theorem~\ref{thm:C3}.
\end{proof}
We record a few relations in ${\mathcal{E}}_3$ that follow from \eqref{eq:a}, \eqref{eq:b},
\eqref{eq:c}, \eqref{eq:d}  by the proof of Theorem~\ref{thm:cov_3}; these relations show the effect of multiplication by
$\mathrm{tr}(t_1t_2t_3)$  on the $P$-module ${\mathcal{E}}_3$ written in the form as in Theorem~\ref{thm:cov_3}:

\begin{proposition}\label{prop:some relations} We have the following equalities:
%\begin{itemize}
%\item[(i)]
{\rm (i)}
\[
\mathrm{tr}(t_1t_2t_3)\cdot t_1=\frac 12\left(\mathrm{tr}(t_1^2)\cdot [t_2,t_3]
-\mathrm{tr}(t_1t_2)\cdot [t_1,t_3]+\mathrm{tr}(t_1t_3)\cdot [t_1,t_2]\right)
\]

%\item[(ii)]
{\rm (ii)} \begin{align*}\label{eq:b}
\mathrm{tr}(t_1t_2t_3)\cdot [t_1,t_2]&=
\frac{1}{4}\left(\mathrm{tr}(t_1t_3)\mathrm{tr}(t_2^2)-\mathrm{tr}(t_1t_2)\mathrm{tr}(t_2t_3)\right)\cdot t_1\\
&+\frac 14 \left(\mathrm{tr}(t_1^2)\mathrm{tr}(t_2t_3)-\mathrm{tr}(t_1t_2)\mathrm{tr}(t_1t_3)\right)\cdot t_2
\\  & +\frac 14\left(\mathrm{tr}(t_1t_2)^2-\mathrm{tr}(t_1^2)\mathrm{tr}(t_2^2)\right)\cdot t_3
\end{align*}

%\item[(iii)]
{\rm (iii)}
\[
\mathrm{tr}(t_1t_2t_3)f_{11}=
\frac 14 \mathrm{tr}(t_1t_3)f_{112}-\frac 14\mathrm{tr}(t_1t_2)f_{113}
-\frac 1{12} \mathrm{tr}(t_1^2)f_{132}+\frac 1{12}\mathrm{tr}(t_1^2)f_{123}
\]

%\item[(iv)]
{\rm (iv)}
\begin{align*}\mathrm{tr}(t_1t_2t_3)f_{112}
&= \frac 12\left(\mathrm{tr}(t_1t_3)\mathrm{tr}(t_2^2)-\mathrm{tr}(t_1t_2)\mathrm{tr}(t_2t_3)\right )f_{11}
\\ & +\frac 12\left(\mathrm{tr}(t_1^2)\mathrm{tr}(t_2t_3)-\mathrm{tr}(t_1t_2)\mathrm{tr}(t_1t_3)\right)f_{12} \\
& +\frac 12 \left(\mathrm{tr}(t_1t_2)^2-\mathrm{tr}(t_1^2)\mathrm{tr}(t_2^2)\right)f_{13}
\end{align*}
%\end{itemize}
\end{proposition}

For an arbitrary $p\ge 3$, denote by $A_p$ the subalgebra of $K[T_p]^{\text{SO}_p(K)}$
generated by $\mathrm{tr}(t_it_j)$, $1\le i\le j\le p$ (note that for $p\ge 4$, $A_p$ is not a polynomial algebra).
The algebra ${\mathcal{E}}_p$ is naturally an $A_p$-module.
Now Theorem~\ref{thm:cov_3} and
Corollary~\ref{cor:weyl} imply the following:

\begin{proposition} \label{prop:A_p}
For any $p\ge 3$, the $A_p$-module ${\mathcal{E}}_p$ decomposes as
\begin{align*}{\mathcal{E}}_p=A_p\cdot I\oplus A_p \cdot \mathrm{tr}(t_1t_2t_3)I \oplus A_p\cdot \langle t_1\rangle_{GL_p(K)}
\oplus A_p\cdot \langle [t_1,t_2]\rangle_{GL_p(K)}
\\
\oplus A_p\cdot \langle t_1^2-\frac 13\mathrm{tr}(t_1^2)I\rangle_{GL_p(K)}
\oplus A_p\cdot \langle [t_1^2,t_2]\rangle_{GL_p(K)}.
\end{align*}
In particular, as an $A_p$-module, ${\mathcal{E}}_p$ is generated by
\[
I,\ \mathrm{tr}(t_1t_2t_3)I,\ t_i,\ t_it_j, \ t_it_jt_k \quad1\le i,j,k\le p.
\]
\end{proposition}

Proposition~\ref{prop:A_p} implies that for $m\ge 4$, any product
$t_{i_1} t_{i_2}\cdots t_{i_m}$ is contained in $A_p^+\cdot {\mathcal{E}}_p$, where
$A_p^+$ stands for the ideal in $A_p$ generated by $\mathrm{tr}(t_it_j)$, $1\le i\le j\le p$.
 A more direct explanation of this fact is given by the following identity:

 \begin{proposition}
 We have the equality
 \begin{align*}
 t_1t_2t_3t_4&=\frac 14\left(\mathrm{tr}(t_1t_4)\mathrm{tr}(t_2t_3)-\mathrm{tr}(t_1t_2)\mathrm{tr}(t_3t_4)\right)I
 \\ & +\frac 12\left(\mathrm{tr}(t_1t_2)t_3t_4-\mathrm{tr}(t_3t_4)t_1t_2-\mathrm{tr}(t_1t_4)t_3t_2\right).
 \end{align*}
 \end{proposition}

\begin{proof}
Proposition~\ref{prop:A_p} implies that $t_1t_2t_3t_4$ must be a $K$-linear combination
of $\mathrm{tr}(t_{\pi(1)}t_{\pi(2)})\mathrm{tr}(t_{\pi(3)}t_{\pi(4)})I$, $\mathrm{tr}(t_{\pi(1)}t_{\pi(2)})[t_{\pi(3)},t_{\pi(4)}]$,
$\mathrm{tr}(t_{\pi(1)}t_{\pi(2)})f_{\pi(3)\pi(4)}$, $\pi\in S_4$.
The actual coefficients were found using \cite{cocoa}:
\begin{align}\label{eq:t1t2t3t4}t_1 t_2 t_3 t_4 &=
 \frac{1}{12}\left(\mathrm{tr}(t_1t_2)\mathrm{tr}(t_3t_4)+\mathrm{tr}(t_1t_4)\mathrm{tr}(t_2t_3) \right)I
 \\ \notag & +\frac 14\left(\mathrm{tr}(t_3t_4)[t_1,t_2]+\mathrm{tr}(t_1t_4)[t_2,t_3]+\mathrm{tr}(t_1t_2)[t_3,t_4]\right)
 \\ \notag & +\frac 12\left(\mathrm{tr}(t_3t_4)f_{12} -\mathrm{tr}(t_1t_4)f_{23}+\mathrm{tr}(t_1t_2)f_{34}\right).
 \end{align}
 Plugging in the explicit expressions for $f_{12}$, $f_{23}$, $f_{34}$ on the right hand side of the above formula, we obtain the desired statement.
\end{proof}

\begin{remark}
Based on Theorem~\ref{thm:cov_3} and its proof,
it is possible to give a normal form for the elements in ${\mathcal{E}}_3$.
With an iterated use of \eqref{eq:t1t2t3t4} it is possible to rewrite the product of any two $P$-module generators of ${\mathcal{E}}_3$ in normal form.
This way one obtains a normal form plus a rewriting algorithm for products of elements given in normal form.
The result is complicated and technical, so we leave out the details.
\end{remark}

\section*{Acknowledgements}

We thank an anonymous referee for several comments improving 
the presentation. In particular, it was the referee's observation that   
the action of the multiplicative group can be used to give a 
direct proof of surjectivity of $\mu^*$ in Lemma~\ref{lemma:a4a5}.


\begin{thebibliography}{AAAAA}

\bibitem[ATZ]{Aslaksen-Tan-Zhu}
H. Aslaksen, E.-C. Tan, C.-B. Zhu,
Invariant theory of special orthogonal groups,
Pac. J. Math. {\bf 168} (1995), No. 2, 207-215.

\bibitem[B]{Berele}
A. Berele,
Homogeneous polynomial identities,
Israel J. Math. {\bf 42} (1982), 258-272.

\bibitem[CoCoA]{cocoa}   CoCoATeam,
  \cocoa: a system for doing
     Computations in Commutative Algebra,
  Available at http://cocoa.dima.unige.it

\bibitem[DoDr]{Domokos-Drensky}
M. Domokos, V. Drensky,
Gr\"obner bases for the rings of invariants of special orthogonal and $2 \times 2$ matrix invariants,
J. Algebra {\bf 243} (2001), 706-716.

\bibitem[Dr1]{Drensky:1}
V. Drensky,
Representations of the symmetric group and varieties of linear algebras (Russian),
Mat. Sb. {\bf 115} (1981), 98-115.
Translation: Math. USSR Sb. {\bf 43} (1981), 85-101.

\bibitem[Dr2]{Drensky:2}
V. Drensky,
Free Algebras and PI-Algebras,
Springer-Verlag, Singapore, 2000.

\bibitem[DrF]{Drensky-Formanek}
V. Drensky, E. Formanek,
Polynomial Identity Rings,
Advanced Courses in Mathematics, CRM Barcelona, Birkh\"auser, Basel-Boston, 2004.

\bibitem[DrK]{Drensky-Koshlukov}
V. Drensky, P.E. Koshlukov,
Weak polynomial identities for a vector space with a symmetric bilinear form,
Math. and Education in Math., Proc. of the 16-th Spring
Conf. of the Union of Bulgar. Mathematicians, Sunny Beach, April 6-10, 1987,
Publishing House of the Bulgarian Academy of Sciences, Sofia (1987), 213-219.
arXiv:1905.08351v1 [math.RA].

\bibitem[F]{Formanek}
E. Formanek,
Central polynomials for matrix rings,
J. Algebra {\bf 23} (1972), 129-132.

\bibitem[K1]{Kaplansky:1}
I. Kaplansky,
Problems in the theory of rings,
Report of a Conference on Linear Algebras, June, 1956,
in National Acad. of Sci.--National Research Council, Washington,
Publ. {\bf 502} (1957), 1-3.

\bibitem[K2]{Kaplansky:2}
I. Kaplansky,
Problems in the theory of rings revised,
Amer. Math. Monthly {\bf 77} (1970), 445-454.

\bibitem[LB]{Le Bruyn}
L. Le Bruyn,
Trace Rings of Generic 2 by 2 Matrices,
Mem. Amer. Math. Soc. {\bf 66} (1987), No. 363.

\bibitem[Mc]{Macdonald}
I.G. Macdonald,
Symmetric Functions and Hall Polynomials,
Oxford Univ. Press (Clarendon), Oxford, 1979, Second Edition, 1995.

\bibitem[P1]{Procesi:1}
C. Procesi,
The invariant theory of $n\times n$ matrices,
Adv. in Math. {\bf 198} (1976), 306-381.

\bibitem[P2]{Procesi:2}
C. Procesi,
Computing with $2\times 2$ matrices,
J. Algebra {\bf 87} (1984), 342-359.

\bibitem[P3]{Procesi:3} 
C. Procesi, Lie Groups (An Approach through Invariants and Representations). 
Springer, New York, 2007. 

\bibitem[Ra1]{Razmyslov:1}
Yu.P. Razmyslov,
Finite basing of the identities of a matrix algebra of second order
over a field of characteristic zero (Russian),
Algebra i Logika {\bf 12} (1973), 83-113.
Translation: Algebra and Logic {\bf 12} (1973), 47-63.

\bibitem[Ra2]{Razmyslov:2}
Yu.P. Razmyslov,  On a problem of Kaplansky (Russian),
Izv. Akad. Nauk SSSR, Ser. Mat. {\bf 37} (1973), 483-501.
Translation: Math. USSR, Izv. {\bf 7} (1973), 479-496.

\bibitem[Ra3]{Razmyslov:3}
Yu.P. Razmyslov,
Finite basis property for identities of representations
of a simple three-dimensional Lie algebra over a field of characteristic zero (Russian),
Algebra, Work Collect., dedic. O. Yu. Shmidt, Moskva 1982, 139-150.
Translation: Transl. Am. Math. Soc. Ser. 2 {\bf 140} (1988), 101-109.

\bibitem[Ra4]{Razmyslov:4}
Yu.P. Razmyslov,
Identities of Algebras and Their Representations (Russian),
``Sovremennaya Algebra'', ``Nauka'', Moscow, 1989.
Translation: Translations of Math. Monographs {\bf 138}, AMS, Providence, R.I., 1994.

\bibitem[Re]{Regev}
A. Regev,
Algebras satisfying a Capelli identity,
Israel J. Math. {\bf 33} (1979), 149-154.

\bibitem[S]{sibirskii} 
K.S. Sibirskii, 
Unitary and orthogonal invariants of matrices (Russian), 
Dokl. Akad. Nauk SSSR {\bf 172} (1967), 40-43. 

\bibitem[W]{Weyl}
H. Weyl,
The Classical Groups, Their Invariants and Representations,
Princeton Univ. Press, Princeton, N.J., 1946, New Edition, 1997.
\end{thebibliography}
\end{document}